\documentclass[11pt]{amsart}
\usepackage[colorlinks=true,citecolor=magenta,linkcolor=magenta]{hyperref}
\usepackage{amssymb,amscd,amsmath,graphicx,mathtools}
\usepackage{enumerate}
\usepackage{fullpage}

\usepackage[all]{xy}
\usepackage{cleveref}

\numberwithin{equation}{section}
\newtheorem{theorem}{Theorem}[section]
\newtheorem{lemma}[theorem]{Lemma}
\newtheorem{proposition}[theorem]{Proposition}
\newtheorem{corollary}[theorem]{Corollary}

\theoremstyle{definition}
\newtheorem{definition}[theorem]{Definition}
\newtheorem{conjecture}[theorem]{Conjecture}
\newtheorem*{acknowledgment}{Acknowledgments}

\newtheorem{setup}[theorem]{Setup}

\theoremstyle{remark}
\newtheorem{remark}[theorem]{Remark}
\newtheorem{example}[theorem]{Example}

\DeclareMathOperator{\Ann}{Ann}
\DeclareMathOperator{\Ass}{Ass}

\DeclareMathOperator{\hgt}{ht}

\newcommand{\NN}{{\mathbb N}}

\def\C{{\mathcal C}}

\def\R{{\mathcal R}}

\def\R{{\mathcal R}}

\def\H{{\mathcal H}}

\def\aa{{\mathfrak a}}
\def\mm{{\mathfrak m}}

\def\pp{{\mathfrak p}}

\def\bb{{\mathfrak b}}
\def\cc{{\mathfrak c}}
\def\dd{{\mathfrak d}}
\def\ee{{\mathfrak e}}

\def\ahat{\widehat{\alpha}}

\def\vhat{\widehat{v}}

\def\1{{\bf 1}}
\def\0{{\bf 0}}

\begin{document}
	
	\title{$F$-Thresholds of Filtrations of Ideals}
	
	\author{Mitra Koley}
	\address{Theoretical Stat-Math Unit, Indian Statistical Institute, 203 B.T Road, Kolkata India-700108}
	\email{mitra.koley@gmail.com}
	%\urladdr{???}

	\author{Arvind Kumar}
	\address{Department of Mathematical Sciences, New Mexico State University, 1780 E University Ave, Las Cruces, US-88003}
	\email{arvkumar@nmsu.edu}

	
	\begin{abstract}
		In this article we extend the notion of  the $F$-thresholds of ideals to the $F$-thresholds for filtrations of ideals. Existence of $F$-thresholds of filtrations are established for various types of filtrations. Moreover various necessary and sufficient conditions for finiteness of $F$-thresholds are given.  We also give some effective upper bounds of the symbolic $F$-thresholds.  A sufficient condition for symbolic $F$-splitting in terms of its symbolic $F$-threshold is given. We compute the symbolic $F$-thresholds of various determinantal ideals, $F$-K\"onig ideals, and many more.
			\end{abstract}
	
	\maketitle

%%%%%%%%%%%%%%%%%%%%%%%%%%%%%%%%%%
%%%%%%%%%%%%%%%%%%%%%%%%%%%%%%%%%%%

\section{Introduction} \label{sec.intro}
The $F$-thresholds were introduced by  Musta\c{t}\u{a}, Takagi and Watanabe  for $F$-finite regular rings as a prime characteristic invariant of singularities \cite{MTW}. This invariant is positive characteristic analogue of log canonical thresholds, an invariant of singularities defined in characteristics $0$. In the introductory paper \cite{MTW}, it was shown that the $F$-thresholds coincide with the jumping exponents for the generalized test ideals defined by Hara and Yoshida in \cite{HY}.
 
 In a subsequent paper by Huneke,  Musta\c{t}\u{a}, Takagi and Watanabe, this notion was generalized for arbitrary rings of prime characteristics \cite{HMTW}. Given a ring $R$ of prime characteristic $p>0$ and ideals $\mathfrak{a}$ and $I$, such that $\mathfrak{a}$ is contained in radical of $I$, the $F$-threshold  asymptotically measures the containment of the powers of $\mathfrak{a}$ in the Frobenius powers of $I$. 
 Given $\mathfrak{a}$ and $I$ as described above, they defined
 for all $q=p^e$, 
 $$\nu^I_{\mathfrak{a}}(q):=\max\{r: \mathfrak{a}^r\not\subseteq I^{[q]}\},$$ where $I^{[q]}=(f^q: f\in I)$ and
 $$\mathcal{C}^I_+(\mathfrak{a}):=\limsup_{e\to\infty}\frac{\nu^I_{\mathfrak{a}}(p^e)}{p^e} \hspace{0.5cm}\textit{and}\hspace{0.5cm} \mathcal{C}^I_-(\mathfrak{a}):=\liminf_{e\to\infty}\frac{\nu^I_{\mathfrak{a}}(p^e)}{p^e}.$$
 When these two limits coincide, the common value is denoted by 
 $\mathcal{C}^I(\mathfrak{a})$ and is called the $F$-threshold of $\mathfrak{a}$ with respect to $I$. In the same article, the existence of the limit is shown for $F$-pure rings.  When the ring is regular the limit always exists as  in this case the sequence $\{\frac{\nu^I_{\mathfrak{a}}(p^e)}{p^e}\}$ becomes a bounded, monotone increasing sequence. Later in \cite{DNP18}, De-Stefani,  N\'{u}\~{n}ez Betancourt and P\'{e}rez
 showed the existence of the $F$-thresholds in full generality.

  Furthermore in \cite{HMTW}, the authors showed its relations  with the Hilbert-Samuel multiplicity, tight closure, and integral closure. Because of its important connections with various areas of commutative algebra and geometry \cite{MMK09, MMK08,  TK15, D09, DKK14,  KMK10, F-threshold-det-ring, BI20, VK21}, this invariant becomes an interesting object to study. The $F$-thresholds of defining ideals of Calabi-Yau hypersurfaces with respect to the homogeneous maximal ideals were computed  by Bhatt and Singh in \cite{BS}. In \cite{Her}, Hern\'andez computed the $F$-threshold of a principle ideal generated by a  binomial a polynomial ring with respect to the homogeneous maximal ideal. In \cite{F-threshold-det-ring}, Miller, Singh, Varbaro computed $F$-thresholds for determinantal ideals with respect to again the homogeneous maximal ideal. Linear programs turn out to be useful in order to compute $F$-thresholds of various classes of ideals in polynomial rings, see \cite{Adam23}. 
  %The main aim of this article is to study the containment behaviour of a filtration of ideals inside Frobenius powers of an ideal. This idea comes from a recent work of H\`a, Kumar, Nguyen and Nguyen see \cite{HKN22} where the authors studied containment behaviour between two filtrations. And to study the containment behaviour of a filtration of ideals inside Frobenius powers of an ideal, we extend the notion of the  $F$-thresholds for filtrations of ideals, study their existence and various kinds of effective bounds. We compute $F$-thresholds of symbolic power filtrations of  various ideals.  We give a numerical criterion for  symbolic $F$-split filtration, a new notion of $F$-singularity defined in \cite{AJL21}, in terms of symbolic $F$-threshold.  Indeed, this work is also inspired by the recent work of {De Stefani}, {Monta{\~n}o}, {N{\'u}{\~n}ez-Betancourt} \cite{AJL21} where they introduced the notion of  $F$-split filtration and its applications. The paper is organized as follows.
  
  The main aim of this article is to extend the notion of the  $F$-thresholds for filtrations of ideals, study their existence and various kinds of effective bounds. We compute $F$-thresholds of symbolic power filtrations of  various ideals. Our work is inspired by \cite{AJL21} and \cite{HKN22}. The idea of studying the containment behaviour of a filtration of ideals inside Frobenius powers of an ideal comes from the recent works of H\`a, Kumar, Nguyen and Nguyen see \cite{HKN22} where the authors studied containment behaviour between two filtrations. In \cite{AJL21}, the authors studied $F$-singularities defined for filtrations, precisely they introduced $F$-split filtrations, a new notion of $F$-singularity defined for filtraions. In this article, we give a numerical criterion in terms of symbolic $F$-threshold for a symbolic $F$-split filtration (Theorem 5.12). The paper is organised as follows.
  
  In Section 2, we  recall  some definitions and useful results on filtrations, valuations, $F$-singularities, and graphs which are required in later sections. We also give a brief discussion on $F$-thresholds of ideals.
  
  In Section 3, we extend the notion of the $F$-thresholds for filtrations of ideals as follows:
  Let $I$ be an ideal of a ring $R$ of prime characteristic $p>0$ and $\mathfrak{a}_{\bullet}= \{\mathfrak{a}_i\}_{i\geq 0}$ be a filtration of ideals in $R$. For every non-negative integer $e$, define $$\nu_{\mathfrak{a}_{\bullet}}^I (p^e) := \sup \{ r \in \mathbb{Z}_{\geq 0} \; : \;  \mathfrak{a}_r \not\subseteq I^{[p^e]} \}.$$
 Define $$\mathcal{C}_+^I(\aa_{\bullet}):=\limsup_{e \to \infty} \frac{\nu_{\aa_{\bullet}}^I(p^e)}{p^e} {\in\mathbb{R}_{\geq 0}\cup\{\pm \infty\}} \text{ and } \mathcal{C}_-^I(\aa_{\bullet}):=\liminf_{e \to \infty} \frac{\nu_{\aa_{\bullet}}^I(p^e)}{p^e}{\in\mathbb{R}_{\geq 0}\cup\{\pm \infty\}}.$$ 
 If $\mathcal{C}_+^I(\aa_{\bullet})=\mathcal{C}_-^I(\aa_{\bullet})\in \mathbb{R}_{\geq 0}$, then we denote it by $\mathcal{C}^I(\aa_{\bullet})$ and call it the \emph{$F$-threshold of $\aa_{\bullet}$ with respect to $I$}. When  $(R,\mathfrak{m})$ is a local ring and $I=\mathfrak{m}$, the maximal ideal of $R$, then $\mathcal{C}^{\mathfrak{m}}(\aa_{\bullet})$ (if exists) is called the \emph{$F$-threshold of $\aa_{\bullet}$}. When $\aa_{\bullet}=\{\aa^{(i)}\}_{i\ge 0}$, a symbolic power filtration, then $\mathcal{C}^{\mathfrak{m}}(\aa_{\bullet})$ is called the \emph{ symbolic $F$-threshold of $J$}.
  We show that any positive real number can be attained as an $F$-threshold of a filtration. After discussing some basic properties of $F$-thresholds for filtrations we compare $F$-thresholds of two filtrations and give conditions for the equality.
   As consequences we get the existence of the $F$-thresholds for integral closure power filtration (for analytically unramified rings), tight closure filtration of an ideal. Rest of the section is devoted to discussions on various conditions for finiteness of $\mathcal{C}_+^I(\aa_{\bullet})$ of any filtration $\aa_{\bullet}$. We show that $F$-thresholds exist for Noetherian filtrations, symbolic power filtrations in polynomial rings.
   
   In Section 4, we identify another class of filtrations for which the $F$-thresholds  with respect to an ideal $I$ exist. We call them the $(I,p)$-admissible filtration. Ordinary power filtrations, symbolic power filtrations and certain integral closure power filtrations appear as examples of $(I,p)$-admissible filtrations.  We also show  that $(I,p)$-admissible filtrations are closed under standard arithmetic operations, i.e., intersections, products and  binomial sums of filtrations and prove related results on the $F$-thresholds. More over for polynomial rings we  compute the $F$-threshold of intersection of two monomial ideals in terms of the $F$-thresholds of the monomial ideals. This result enable us to compute  $F$-thresholds of various classes of ideals in a polynomial ring. 
   
   In Section 5 we study $F$-thresholds of symbolic power filtrations deeply. we show that for any regular ring $R$,  the $F$-threshold of  an symbolic power filtration of an ideal $I$ is bounded above by 
   big-height of  $I$. We show that for various determinantal ideals, $F$-K\"onig ideals, the symbolic $F$-thresholds are height of the respective ideals. More generally we give an equivalent criterion for $F$-thresholds for symbolic power filtration of an ideal $\aa$ with respect to an ideal $I$, to be big-height of $\aa$ in terms of ideal containment.
   We show that if the symbolic $F$-threshold of a radical ideal in an $F$-finite regular local ring is big-height of that ideal, then the ideal is unmixed and symbolic $F$-split. We conjecture that the converse is also holds.
   
   \begin{conjecture}
    Let $(R,\mathfrak{m})$ be an $F$-finite regular local ring and $I$ be an unmixed radical ideal. Suppose $I$ is symbolic $F$-split, then $\mathcal{C}^{\mathfrak{m}}(I^{(\bullet)})=\text{big-height of } I$.
   \end{conjecture}

   It is known from \cite{AJL21}, determinantal ideals corresponding to generic, Pfaffian, symmetric, Hankel matrices, unmixed $F$-K\"onig ideals, squarefree monomial ideals are symbolic $F$-split. Hence the conjecture is true for those ideals. One can not replace the symbolic $F$-split by $F$-split, as we show by an example that the conjecture is false if we assume $R/I$ is only $F$-split 
   
   In section 6, we focus our attention to $F$-thresholds of homogeneous filtrations with respect to the maximal monomial ideal in a polynomial ring and its interlinkage with valuations. Our first contribution of this section is to give an equivalent criterion of boundedness of $F$-thresholds of any homogeneous filtration in terms of positiveness of Skew Waldschmidt constant of any monomial valuation. We also give various examples of ideals (both valuations and graph theoretic) where the bounds  are attained. Next we give a formula for  $F$-threshold of any monomial ideal in terms of its Rees valuations. We  also give a bound of $F$-threshold of any monomial filtration in terms of valuations. Our next contribution is finding a formula for $F$-threshold of symbolic power filtration of any square-free monomial ideal $I$  in terms of its monomial valuations associated to its minimal primes and the value is equal to the height of $I$.
 
 \begin{acknowledgment} 
The first author was supported by DST, Govt. of India under the DST-INSPIRE Faculty Scheme (Ref no: DST/INSPIRE/04/2020/001266). 
\end{acknowledgment}

 \section{Preliminaries}
 In this section we recall some definitions and facts which we use in later sections. All rings considered in this article are Noetherian commutative ring with unity. %We begin with recalling singularities in prime characteristic.
 
 \subsection{Filtrations and Rees Algebras:}
 Let $R$ be a ring. A \emph{filtration} of ideals in $R$ is a sequence of ideals $\{\mathfrak{a}_i\}_{i \ge 0}$ that satisfy the followings:
 \begin{enumerate}
  \item $\mathfrak{a}_0=R$,
  \item  $\mathfrak{a}_{i+1}\subseteq {\mathfrak{a}_{i}}$ for all $i \ge 0$,
  \item  $ \mathfrak{a}_i\mathfrak{a}_j\subseteq \mathfrak{a}_{i+j}$ for all $i,j \ge 0$.
 \end{enumerate}
We use $\aa_\bullet$ to denote the filtration $\{\mathfrak{a}_i\}_{i \ge 0}$. Let $R$ be a graded ring, a filtration $\aa_{\bullet}$ is called \emph{homogeneous}, if for each $i$, $\aa_i$ is a homogeneous ideal of $R$.

 Well known examples of filtrations arise from a non-zero proper ideal in $R$. In the following, we illustrate a few standard examples  of filtrations.
 \begin{example} 
Let $\mathfrak{a}, \bb$ be non-zero proper ideals in $R$.
  \begin{enumerate}
  \item The most well known and well studied example of a filtration is the $\aa$-adic filtration. We use $\aa^{\bullet}$ to denote the $\aa$-adic filtration, i.e., $\mathfrak{a}^\bullet=\{\mathfrak{a}^i\}_{i\ge 0}$.
  \item The  integral closure of  an ideal $\aa$, denoted by $\overline{\aa}$, is defined as the ideal  $\overline{\aa}:=(x\in R : \text{ there exist } n \in \NN \text{ and } a_i \in \aa^i \text{ for } 0 \le i \le n \text{ such that }
  x^n+a_1 x^{n-1}+ \cdots + a_n= 0).$ An ideal $\aa$ is called \emph{integrally closed}  if $\aa=\overline{\aa}$.
  Integral closure of powers of $\aa$ forms a filtration and it is known as the \textit{normal filtration} of $\aa$. We use $\overline{\mathfrak{a}^{\bullet}}$ to denote it, i.e., $\overline{\mathfrak{a}^\bullet}=\{\overline{\mathfrak{a}^i}\}_{i\ge 0}$.
  \item \textit{Saturation of powers of $\aa$ with respect to $\bb$} forms a filtration, see \cite{HJKN2023}. We use $\aa^{(\bullet)}_{\bb}=\{\aa^i: \bb^{\infty}\}_{i\ge 0}$, the notation from \cite{HJKN2023}, to denote this filtration. Here $\aa^{(i)}_\bb =\mathfrak{a}^i:\bb^{\infty}=\cup_{n\ge 1} (\aa^i : \bb^n)$. In \cite[Lemmas 2.1, 2.2]{HJKN2023}, the authors proved that both the notions of symbolic powers of $\aa$ can be realized as the saturated powers of $\aa$ with respect to appropriate ideals. 
  \item The $i$-th symbolic power of $\aa$, denoted by $\aa^{(i)}$, is defined as $\cap_{\pp \in \Ass(\aa)} (\aa^iR_{\pp} \cap R).$ As said above,  \textit{symbolic powers} of $\aa$ forms a filtration. We denote this filtration by $\aa^{(\bullet)}=\{\aa^{(i)}\}_{i \ge 0}.$
  %\item Assume that characteristic of $R$ is prime $p$. Tight closure of powers of $\aa$ forms a filtration of ideals in $R$. We use $(\aa^{\bullet})^*=\{(\aa^i)^*\}_{i\ge 0}$ to denote this filtration, and we call it as \textit{tight closure filtration} of $\aa$. 
    \item Let $f : \NN \to \mathbb{R}_{>0}$ be a subadditive sequence. Then $\mathfrak{a}_n:=\mathfrak{a}^{\left\lceil f(n) \right\rceil}$ is a filtration of ideals.
 \end{enumerate}
 \end{example}
 
Let $\aa_\bullet$ be a filtration of ideals in $R$. It is easy to note that $\sqrt{\aa_i}=\sqrt{\aa_1}$ for all $i \ge 1$. Therefore, we call $\sqrt{\aa_1}$ to be the \emph{radical of} $\aa_\bullet$, and denote it by $\sqrt{\aa_\bullet}.$ Similar to Rees algebra associated to any ideal in a ring, one can  also define  Rees algebra  associated to any filtration $\aa_\bullet$ of ideals. The \emph{Rees Algebra} of $\aa_\bullet$ is defined as $$\mathcal{R}(\aa_\bullet):= \bigoplus_{i \ge 0} \aa_i t^i \subseteq R[t].$$ We say that a filtration $\aa_\bullet$ is  \emph{Noetherian} if the asociated Rees Algebra $\mathcal{R}(\aa_\bullet)$ is Noetherian. For $k \in \NN$, the \textit{$k$-th Veronese subalgebra} of $\R(\aa_\bullet)$ is defined to be $$\R^{[k]}(\aa_\bullet) := \bigoplus_{r \ge 0} \aa_{kr} t^{kr} \subseteq R[t].$$  It is well known that, see \cite[Proposition 2.1]{P88}, if $\aa_\bullet$ is a Noetherian filtration, then $\R^{[k]}(\aa_\bullet)$ is a standard graded $R$-algebra for some $k$.

 \subsection{Valuations}
  
 Here we review some basic facts about valuations.  Our basic reference for valuations is \cite{HS06}.
 
 Let $K$ be a field and $G$ be a totally ordered Abelian group (written additively). A valuation on $K$
is a group homomorphism $v :K \setminus \{0\} \to G\cup \{\infty\}$ such that 
\begin{enumerate}
 \item $v(0)=\infty$, and
 \item for all $x$ and $y$ in $K$, $v(x + y) \geq  \min\{v(x), v(y)\}$.
\end{enumerate}

Let $R$ be a domain and $K$ denote its field of fractions. A valuation $v$ on $K$ is \emph{supported on $R$}  if $v(x)\geq 0$ for all $x\in R\setminus \{0\}$. Here in this article we only consider  valuations on the fields of fractions of polynomial rings with $G=\mathbb{R}$, i.e. only real valuations. A valuation $v$ on the field of fractions $K$ of a polynomial ring $R$ is said to be a \emph{monomial valuation}, if for every non-zero polynomial 
$f \in R$,  $v(f )=\min\{  v(u) : u \text{ is a monomial term in } f\} $. 

Since in this article we only consider valuations on the fields of fractions of polynomial rings, in the following we recall the definition of Rees valuations of an ideal in a domain.

  Let $R$ be a domain with field of fractions $K$ and $I$ be an ideal in $R$. A set of \emph{Rees valuation domains} of $I$ is a set $\{V_1, \cdots, V_s\}$ consisting of valuation domains, with the following conditions:
  \begin{enumerate}
   \item For every $i$, $R\subseteq V_i\subseteq K$.  Moreover, each $V_i$ is Noetherian and it is not a field.
   \item  For each $n \in \mathbb{N}$, 
   $\overline{I^n} = \bigcap_{i=1}^s(I^nV_i) \cap R$.
   \item The set $\{V_1, \cdots, V_s\}$ satisfying the previous conditions is minimal possible.
   \end{enumerate}
The valuations $\{v_1,\cdots,v_s\}$ that correspond to valuation domains $V_1,\cdots,V_s$ respectively are called a set of \emph{Rees valuations} of $I$.

\begin{theorem}
  \begin{enumerate}
  \item \cite[Theorem 10.1.6, Theorem 10.2.2]{HS06} Every ideal $I$ in a Noetherian domain has a unique set (up to equivalence) of Rees valuations.
   \item  \cite[Proposition 10.3.4]{HS06} Let $I$ be a monomial ideal in a polynomial ring. Then every Rees valuation of $I$ is a monomial valuation.
   %\item
  \end{enumerate}

 \end{theorem}
\begin{remark}
\label{valuative-criterion}
 Let $R$ be a domain with field of fractions $K$ and $I$ be an ideal in $R$. For any valuation $v$ on $K$, we denote $v(I):=\min \{v(x): x\in I\setminus \{0\}\}$. Let $\{v_1, \cdots, v_s\}$ be the Rees valuations of $I$, then for each $n>0$, integral closure of $I^n$ is given by
 $$\overline{I^n} = \{r\in R : v_i(r) \geq nv_i(I), \text{ for  all } 1\leq i\leq s\}.$$
\end{remark}

Next we recall the definition of (skew) Waldschmidt constant of a filtration from  \cite{HKN22}. 
 Let $R$ be a graded domain and $K$ denote its field of fractions.
    Let $v$ be a valuation on $K$ which is supported on $R$.  %For any ideal $I$ of $R$, set
%$$v(I) = \min \{v(x) ~\big|~ x \in I \setminus \{0\}\}.$$
Let $\aa_{\bullet}$ be a filtration of ideals in $R$.
 Then \emph{skew Waldschmidt constant} of $\aa_\bullet$ with respect to $v$ is defined to be
$$\vhat(\aa_\bullet):= \lim\limits_{n \rightarrow \infty} \dfrac{v(\aa_n)}{n} = \inf\limits_{n \in \NN} \dfrac{v(\aa_n)}{n}.$$ 
It can be seen that for any filtration  $\aa_\bullet$, $\{v(\aa_i)\}_{i \ge 1}$ is a non-negative sub-additive sequence. Thus, by Fekete's Lemma, the last equality of the definition follows.

Assume $R$ is standard graded. Let $\aa_{\bullet}$ be a homogeneous filtration of ideals in $R$ and $\alpha$ denote the valuation on $K$,  defined  as follows: let $f=f_{d_1}+\cdots+f_{d_k}$ be a polynomial in $R$, where each $i$, $f_{d_i}$ denotes its nonzero degree $d_i$-th homogeneous part, define $\alpha(f):=\min_i d_i$. Then $\alpha$ is called the \emph{degree valuation} and $\hat{\alpha}(\aa_\bullet)$
is called the \emph{Waldschmidt constant} of the  filtration $\aa_{\bullet}$.

\begin{remark}
 \begin{enumerate}
  \item One can note that $0 \le \vhat(\aa_\bullet) \le \frac{v(\aa_n)}{n} \le v(\aa_1).$
  \item In general, the skew-Waldschmidt constant of a filtration $\aa_\bullet$ with respect to a valuation $v$ need not be positive, for example, see \cite[Example 2.8]{HKN22}.
 \end{enumerate}

\end{remark}

  \subsection{Singularities in prime characteristics and $F$-thresholds}
Let $R$ be a Noetherian ring of prime characteristic $p>0$, and let $F : R \to R$ denote the Frobenius homomorphism that sends $r\mapsto r^p,$  for any $r\in R$.
A ring $R$ of prime characteristic $p>0$ is called \emph{$F$-finite} if $R$ becomes a finite $R$-module via the Frobenius map. For a reduced ring $R$, let $R^{1/q}$ denote the ring  of all $q$-th roots of $R$. One can note that a reduced ring $R$ is $F$-finite if and only of $R^{1/p}$ is module finite over $R$.

For  every positive integer $e$, we write $q=p^e$. For any $R$-ideal $I=(f_1,\cdots,f_n)$, let $I^{[q]}:=(f_1^q,\cdots, f_n^q)$.

 In the following we recall a few definition of singularities in prime characteristic that are defined in terms of action of the Frobenius.
 \begin{definition}
  Let $R$ be a ring of prime characteristics $p>0$ and $I$ be an ideal of $R$. 
  \begin{enumerate}
   \item The \emph{Frobenius closure} of $I$,  is the ideal, denoted by $I^{F}$,  formed by the elements $x\in R$ such that $x^q\in I^{[q]}$ for all $q=p^e\gg 0$. An ideal $I$ is called \emph{Frobenius closed} if $I^F=I$.
   \item  The \emph{tight closure} of $I$,  is the ideal denoted by $I^{*}$ and formed by the elements $x\in R$ such that there exists a $c\in R\setminus {\cup}_{\mathfrak{p}\in Min (R)} \mathfrak{p}$ and  $cx^q\in I^{[q]}$ for all $q=p^e\gg 0$. 
   An $R$-ideal  $I$ is \emph{tightly closed} if $I^*=I$. One notes that $I\subseteq I^F\subseteq I^*\subseteq \overline{I}$.
   Tight closure of powers of  an ideal $\aa$ forms a filtration of ideals in $R$. We use $(\aa^{\bullet})^*=\{(\aa^i)^*\}_{i\ge 0}$ to denote this filtration, and we call it as the \emph{tight closure filtration} of $\aa$.
   
   \item The ring $R$ is called \emph{$F$-pure}  if the Frobenius map is a pure homomorphism, that is for any $R$-module $M$, $M\otimes_R R\xrightarrow{1_M\otimes F} M\otimes_R R$ is injective.
   
   \item The ring $R$ is called \emph{weakly $F$-regular} if  all $R$-ideals are tightly closed.  An $F$-finite reduced ring $R$ is called \emph{strongly $F$-regular} if  for every element $c$ of $R$ which is not any minimal prime of $R$, there exists a $e>0$, such that such that the $R$-module map $R\to R^{1/q}$ sending $1\mapsto  c^{1/q}$ splits as a map of $R$-modules, where $q=p^e$. Every strongly $F$-regular ring is $F$-split.
   \item \cite[Definition 5.2]{AJL21} Assume further that $R$ is  $F$-finite, regular which is either local or $\mathbb{Z}_\geq 0$ graded. Then $I$ is called \emph{symbolic $F$-split} if  for every $n\in \mathbb{Z}_{\geq 0}$, there exists a splitting $\phi: R^{1/p}\to R$ such that 
 $\phi((I^{(np+1)})^{1/p})\subseteq I^{(n+1)}$.
  \end{enumerate}
    
  \end{definition}
In the following we summarize some basic results from literature.
 \begin{theorem}
  \begin{enumerate}
   \item A ring is regular if and only if the Frobenius map $F:R\to R$ is flat. In particular, in a regular ring every ideal is tightly closed.
   \item An $F$-finite regular ring is strongly $F$-regular, hence it is $F$-split.
   \item In an $F$-pure ring, every ideal is Frobenius closed.
   \item \cite{AJL21} Every symbolic $F$-split ideal is $F$-split. But the converse is false \cite[Example 5.13]{AJL21}.
  \end{enumerate}

 \end{theorem}
 
At the end we briefly discuss about $F$-thresholds for ideals in a ring.
 \begin{definition}{\cite{HMTW}}
 \label{F-thresholds}
 Let $R$ be a ring and $\mathfrak{a}$ and $I$ be two $R$-ideals, such that $\mathfrak{a}\subseteq \sqrt I$.
 For all $q=p^e$, define, 
 $$\nu^I_{\mathfrak{a}}(q):=\max\{r: \mathfrak{a}^r\not\subseteq I^{[q]}\},$$
 and  the limit
 $$\mathcal{C}^I(\mathfrak{a}):=\lim_{e\to\infty}\frac{\nu^I_{\mathfrak{a}}(q)}{q} \text{ is called the \emph{$F$-threshold} of $\mathfrak{a}$ with respect to $I$}.$$
  \end{definition}
  \begin{remark}
  \begin{enumerate}
  \item \cite[Remark 2.1]{HMTW}
   Note that since $\aa\subseteq \sqrt I$,  for each $q$, $\nu^I_{\mathfrak{a}}(q)$ a non negative integer.  If $\mu(\aa)$  is the number of elements in a generating set of $\aa$  and $\aa^N\subseteq I$ for some $N>0$, one has
    $$\nu^I_{\mathfrak{a}}(q)\leq N(\mu(\aa)(q -1) + 1) -1.$$
    Hence the sequence $\{\frac{\nu^I_{\mathfrak{a}}(p^e)}{p^e}\}$ is a bounded sequence and bounded by $N\mu(\aa)$.
  \item The notion of the $F$-threshold of an ideal was first introduced for an $F$-finite regular ring by  Musta\c{t}\u{a}, Takagi and Watanabe \cite{MTW}. For regular rings, the sequence $\{\frac{\nu^I_{\mathfrak{a}}(q)}{q}\}$ becomes a monotone bounded sequence, hence the limit exists \cite[Lemma 1.1 and Remark 1.2]{MTW}.
  Later the definition was generalized for arbitrary rings by Huneke, Musta\c{t}\u{a}, Takagi and Watanabe in \cite{HMTW}. In that paper the authors prove the existence of the limit for $F$-pure rings. Later in \cite{DNP18}, De Stefani, N\'{u}\~{n}ez Betancourt and P\'{e}rez proved the existence of $C^J(\mathfrak{a})$ for any arbitrary ring $R$ \cite[Theorem 3.4]{DNP18}.

  \end{enumerate}
\end{remark}

  \subsection{Graphs, Hypergraphs and  associated ideals:} A \emph{hypergraph} $\mathcal{H}$ is a pair $(V (\H), E(\H))$ such that $V(
\H)$ is a non-empty finite set and $E(\H)$ is a non-empty collection of subsets of $V(\H)$. An element of $V(\H)$ is called a \emph{vertex} of $\H$ and an element of $E(\H)$ is called an \emph{edge} of $\H$. We call $\mathcal{H}$ a \emph{simple hypergraph} if there are no subset inclusions between the edges of $\mathcal{H}$. 
A subset ${W} \subseteq V(\H)$ is called a \emph{vertex cover} of $\H$, if  ${W}\cap e\neq \emptyset$ for every edge $e$. A subset  $M \subseteq E(\H)$ is called a \emph{matching} if  every pair of distinct edges in $M$ are disjoint.
The \emph{matching number}, denoted by $m(\H)$, is the maximum size of matchings in $\H$. A matching $M$  in a hypergraph $\H$ is called \emph{perfect} if every vertex of the hypergraph is in exactly one edge of $M$.
 A fractional matching is a function $f: E(\H)\to [0,1]$ such  that, for each vertex $v$, we have $\sum f (e) \leq 1$ where the sum is taken over all edges $e$ incident to $v$. \emph{The fractional matching number} of a graph $G$ is the $\sup_{f}\{\sum_{e\in E(\H)} f(e): \text{where } f \text{ is a fractional matching}\}$.
The \emph{chromatic number} of a hypergraph $\H$ is the minimum number of colors required to color vertices of $\H$ so that adjacent vertices have different colors.
A \emph{$b$-fold coloring} of a hypergraph $\H$ assigns
to each vertex of $\H$ a set of $b$ colors so that adjacent vertices have different colors. We say that $\H$ is a \emph{$b$-colorable} if it has a $b$-fold coloring.  The minimum $a$ for which $\H$ has a $b$-fold coloring is the $b$-fold chromatic number of $\H$, denoted $\chi_b(\H)$.  Define the fractional chromatic number to be
 $$\chi_f(\mathcal{H}):= \lim_{b\to \infty} \frac{\chi_b(\H)}{b}=\inf_b \frac{\chi_b(\H)}{b},$$
 where the equality occurs because $\chi_{a+b}(\H)\leq \chi_a(\H)+\chi_b(\H)$.
 
Let $\H$ be a hypergraph with vertex set $V(\H) = \{1, . . . , n\}$, and let $R=\mathbb{K}[x_1, \cdots, x_n]$, where $\mathbb{K}$ is a field. The \emph{edge ideal} of the hypergraph $\H$ is the square-free monomial ideal in $R$ defined as 
  $$I_{\mathcal{H}}: =\left(\prod_{i\in e}x_i: e\in E(\H)\right)$$
 and the \emph{cover ideal} of  the hypergraph $\H$ is the square-free monomial ideal in $R$ defined as  
  $$J({\mathcal{H}}): =\left(\prod_{i\in W}x_i: W \text{ is a vertex cover of } \H \right).$$ The edge ideal of a hypergraph is an Alexander dual of the cover ideal of that hypergraph and vice versa, i.e., $$I(\H)= \bigcap_{W \text{ is a vertex cover }} (x_i : i \in W) \text{ and }J(\H)=\bigcap_{e\in E(\H)} (x_i : i \in e). $$
 The square-free monomial ideals are in a one to one correspondence with simple hypergraphs via edge ideals.

  A hypergraph with each edge of cardinality two is known as \emph{graph}. Let $G=(V(G),E(G))$ be a graph.  A graph $H=(V(H),E(H))$ is  called a \emph{subgraph} $G$ if $V (H) \subseteq V (G)$ and $E(H)\subseteq E(G)$. A subgraph $H$ is said to be an \emph{induced subgraph of} $G$ if $E(H)=\{ \{v,w\} \in E(G) :  v, w \in V (H)\}$. 
 
 An \emph{$n$-cycle}, $C_n$, is a graph with $n$ vertices $\{v_1,\ldots,v_n\}$ and  with edges $\{\{v_i,v_{i+1}\}: 1\leq i\leq n-1\}\cup \{v_n,v_1\}$.  An $n$-cycle is called an \emph{odd cycle} (resp. \emph{even}) if $n$ is odd (resp. even).  A graph is said to be \emph{chordal} if there is no induced cycle on more than three vertices. 
A \emph{path}, $P_n$, of length $n-1$ is a graph with $n$ vertices $\{v_1,\ldots,v_n\}$ and  edges $\{\{v_i,v_{i+1}\}: 1\leq i\leq n-1\}$. A \emph{complete} graph, $K_n$, is a graph with $n$ vertices $\{v_1,\ldots,v_n\}$ and edges $E(G)=\{\{v_i,v_j\}: 1 \le i< j\le n \}$.  An induced complete subgraph of a graph $G$ is called a \emph{clique}. The \emph{clique
number} of a graph $G$, denoted by $\omega(G)$, is the maximum size of cliques of $G$. A graph $G$ on $n$ vertices is said to be \emph{traceable}, if $P_n$ is a subgraph of $G$, and is said to be  \emph{Hamiltonian} if $C_n$ is a subgraph of $G$. A graph is called \emph{vertex transitive} if given any two vertices $v$ and $w$ of G, there is an automorphism $f: G\to G$ such that $f(v)=w$.

   %Let $G$ be a simple graph on $n$ vertices (i.e. without multiple edges and without loops). Then $G$ is said to be \emph{closed} if there exists a labeling on vertices of $G$ such that for all $1\leq i<k<j\leq n$, if $\{i,j\}\in E(G)$, then both $\{i,k\}\in E(G)$ and 
   %$\{k,j\}\in E(G)$.  The graph $G$ is said to be \emph{weakly closed} if there exists a labeling on vertices of $G$ such that for all $1\leq i<k<j\leq n$, if $\{i,j\}\in E(G)$, then  either $\{i,k\}\in E(G)$ or 
   %$\{k,j\}\in E(G)$. The \emph{binomial edge ideal} of a graph $G$ is a binomial ideal defined as follows
  %$$\mathcal{J}_G:=(x_iy_j-x_jy_i: \{i,j\}\in E(G)) \subseteq \mathbb{K}[x_1,\ldots,x_n,y_1\ldots,y_n] .$$

%%%%%%%%%%%%%%%%%%%%%%%%%%%%%%%%%
%%%%%%%%%%%%%%%%%%%%%%%%%%%%%%%%%

 \section{$F$-thresholds of filtrations: their properties and existence}
 In this section, we extend the notion of $F$-thresholds of ideals to the notion of $F$-thresholds of  filtrations of ideals and study their basic properties.  Throughout this article, $R$ denote a Noetherian commutative ring of {prime} characteristic $p>0$,
 $I$ is an ideal of $R$ and $\mathfrak{a}_{\bullet}= \{\mathfrak{a}_i\}_{i\geq 0}$ is a filtration of ideals in $R$. For every non-negative integer $e$, we define $$\nu_{\mathfrak{a}_{\bullet}}^I (p^e) := \sup \{ r \in \mathbb{Z}_{\geq 0} \; : \;  \mathfrak{a}_r \not\subseteq I^{[p^e]} \}.$$

We define $$\mathcal{C}_+^I(\aa_{\bullet}):=\limsup_{e \to \infty} \frac{\nu_{\aa_{\bullet}}^I(p^e)}{p^e}{\in\mathbb{R}_{\geq 0}\cup\{\pm \infty\}} \text{ and } \mathcal{C}_-^I(\aa_{\bullet}):=\liminf_{e \to \infty} \frac{\nu_{\aa_{\bullet}}^I(p^e)}{p^e}{\in\mathbb{R}_{\geq 0}\cup\{\pm \infty\}}.$$ {If} $\mathcal{C}_+^I(\aa_{\bullet})=\mathcal{C}_-^I(\aa_{\bullet}){\in \mathbb{R}_{\geq 0}}$, then we denote it by $\mathcal{C}^I(\aa_{\bullet})$ and call it the \emph{$F$-threshold of $\aa_{\bullet}$ with respect to $I$}. {When $(R,\mm)$ is a local ring, for simplicity, we call $\C^{\mm}(\aa_\bullet)$ to be the \emph{$F$-threshold of $\aa_\bullet$}   and $\mathcal{C}^{\mathfrak{m}}(\aa^{(\bullet)})$ to be the \emph{symbolic $F$-threshold of $\aa$.}

 Throughout this article, we assume that $I$ is a non-zero proper ideal of $R$ as the trivial cases are discussed in the following remark:
 \begin{remark}\label{rmk}
 If $I =R$, then $\mathfrak{a}_r \subseteq I^{[p^e]}$  for all $ r, e \geq 0$. Thus,  $\mathcal{C}^I_{\pm}(\aa_\bullet)=-\infty.$ Assume that if $I=(0)$. If $\sqrt \mathfrak{a}_\bullet \neq (0)$, then $\mathfrak{a}_r \not\subseteq I^{[p^e]}$ for all $r,e \geq 0$. Hence, $\mathcal{C}^I_{\pm}(\aa_\bullet)=\infty.$ And if $ \mathfrak{a}_r = (0)$ for $r \gg 0$, then $\nu_{\mathfrak{a}_{\bullet}}^I (p^e) =k$ for some $k$. Thus, $\mathcal{C}^I_{\pm}(\aa_\bullet)=0.$ 
 
 Suppose that $I$ is proper. In this case $0 \le \nu_{\mathfrak{a}_{\bullet}}^I (p^e) \le \infty$, and therefore, $0 \le \mathcal{C}_{-}^I(p^e) \le \mathcal{C}_{+}^I(p^e) \le \infty$. 
 \end{remark}
 
{Next, we show that any non-negative real number can be realized as  F-threshold of a filtration in a ring $R$. Indeed, given any non-negative real number $\alpha$, we provide a homogeneous filtration $\aa_\bullet$ in a polynomial ring $R$ such that $\mathcal{C}^{\mathfrak{m}}(\aa_\bullet) =\alpha$, where $\mathfrak{m}$ denotes the unique homogeneous maximal ideal of $R$}.  

\begin{example}
\label{alpha}
Let $\alpha$ be a {non-negative} real number. Let $R=\mathbb{K}[x_1,\ldots,x_n]$, { where $\mathbb{K}$ is a field of characteristic $p>0$} and $\mathfrak{m}=(x_1,\ldots,x_n)$.  {When $\alpha=0$, then one can take $\aa_{\bullet}$ to be the zero filtration, i.e., $\aa_i=0$ for $i\ge 1$. Then, $\mathcal{C}^{\mathfrak{m}}(\aa_\bullet) =0$. Next, assume that $\alpha>0$.}   Take $\aa_r= \mathfrak{m}^{\left\lceil \frac{n r}{\alpha} \right\rceil} $ for every $r \in \mathbb{N}$. It is easy to see that $\aa_\bullet$ is a filtration, see also \cite[Example 2.14]{HKN22}. We claim that $ \mathcal{C}^{\mathfrak{m}}(\aa_\bullet) =\alpha$. If $r \ge \alpha p^e$, then $\frac{r n}{\alpha} \ge np^e$ which implies that $\aa_r \subseteq \mathfrak{m}^{np^e} \subseteq \mathfrak{m}^{[p^e]}$, where the last containment follows {from  the Pigeonhole Principle}, (cf. \cite[Lemma 2.4(1)]{HH02}). 
Therefore, $\nu_{\aa_\bullet}^{\mathfrak{m}}(p^e) < \alpha p^e$ for all $e$, and hence, $\mathcal{C}_+^{\mathfrak{m}}(\aa_\bullet) =\limsup_{e \to \infty} \frac{\nu_{\aa_\bullet}^{\mathfrak{m}}(p^e)}{p^e} \le \alpha.$  Next, take $r=\lceil \alpha (p^e-1) \rceil -1.$ Then, $\lceil\frac{rn}{\alpha}\rceil \le n(p^e-1),$ and therefore, $\aa_r \not\subseteq \mathfrak{m}^{[p^e]}$ as $(x_1 \cdots x_n)^{p^e-1} \in \aa_r$.  Thus, for all $e$, $\nu_{\aa_\bullet}^{\mathfrak{m}}(p^e) \ge \lceil \alpha (p^e-1) \rceil -1.$  Consequently, $$\frac{\nu_{\aa_\bullet}^{\mathfrak{m}}(p^e)}{p^e}  \ge \frac{\lceil \alpha (p^e-1) \rceil -1}{p^e} \ge \frac{\alpha (p^e-1)-1} {p^e}.$$ Thus, by definition of liminf, we get $\mathcal{C}_-^{\mathfrak{m}}(\aa_\bullet) \ge \alpha.$ Hence, $\mathcal{C}^{\mathfrak{m}}(\aa_\bullet)$ exists and it is equal to $\alpha.$
\end{example}

 Next, we study some basic properties of $F$-threshold of filtration. {Before that we recall the definition of cyclic pure extension. A ring extension  $R\to S$ is said to be \emph{cyclic pure}, if for every ideal $J$ in $R$,  $JS\cap R = J$.}
 
\begin{proposition}({cf. \cite[Proposition 2.2]{HMTW}})
\label{basic-prop} Let $I, J$ be proper {non-zero} ideals and $\aa_{\bullet}$ be filtration in $R$. Then, we have the followings: 
\begin{enumerate}
\item  Let $e_0\geq 0$ be a fixed integer. If $I^{[p^e]}=(I^{[p^e]})^F$ for all $e\ge e_0$, then  $p \nu_{\mathfrak{a}_{\bullet}}^I (p^e) \leq \nu_{\mathfrak{a}_{\bullet}}^I (p^{e+1})$ for all $e \ge e_0$. Moreover, $${\mathcal{C}^I_{\pm}(\aa_\bullet)}=\lim_{e \to \infty} \frac{\nu_{\mathfrak{a}_{\bullet}}^I (p^e)}{p^e} = \sup_{e \ge e_0} \frac{\nu_{\mathfrak{a}_{\bullet}}^I (p^e)}{p^e}.$$ 
{In particular, when $R$ is $F$-pure,}
$${\mathcal{C}^I_{\pm}(\aa_\bullet)}=\lim_{e \to \infty} \frac{\nu_{\mathfrak{a}_{\bullet}}^I (p^e)}{p^e} = \sup_{{e \ge 0}} \frac{\nu_{\mathfrak{a}_{\bullet}}^I (p^e)}{p^e}.$$

\item 
If $\sqrt{\aa_\bullet} \cap R^{o} \neq \emptyset$, $I^{[p^e]}=(I^{[p^e]})^*$ for $e\ge e_0$ and $0<{\lim_{e \to \infty} \frac{\nu_{\mathfrak{a}_{\bullet}}^I (p^e)}{p^e}} <\infty$, then $\frac{\nu_{\aa_{\bullet}}^I(p^e)}{p^e}< \mathcal{C}^I(\aa_\bullet)$ for $e \ge e_0$.
\item If $I \subseteq J$, then $\mathcal{C}^J_{\pm}(\aa_\bullet) \le \mathcal{C}^I_{\pm}(\aa_\bullet)$.
\item  For any $q=p^{e_0}$, $\mathcal{C}^{I^{[q]}}_{\pm}(\aa_\bullet) =q \;\mathcal{C}^I_{\pm}(\aa_\bullet)$.
\item  If $R\to S$ is a cyclic pure extension, {i.e., for every ideal $J$ in $R$, we have $JS\cap R = J$,} then $C^I(\mathfrak{a}_{\bullet})=C^{IS}(\mathfrak{a}_{\bullet}S),$ where $\aa_\bullet S =\{\aa_i S\}_{i\ge 0}$.
\end{enumerate}
\end{proposition}
\begin{proof}
\begin{enumerate}
\item Let $e \ge e_0$ be any. Suppose that $\aa_r \not\subseteq I^{[p^e]}$. Since $I^{[p^e]}=(I^{[p^e]})^F$ for $e\ge e_0$, we have  $\aa_r^{[p]} \not\subseteq I^{[p^{e+1}]}$ which implies that $\aa_r^p \not\subseteq I^{[p^{e+1}]}$. Now, as $\aa_{r}^p \subseteq \aa_{rp}$, we have $\aa_{rp} \not\subseteq I^{[p^{e+1}]}.$ Thus, $p  \nu_{\mathfrak{a}_{\bullet}}^I (p^e) \leq \nu_{\mathfrak{a}_{\bullet}}^I (p^{e+1})$ for $e \ge e_0$. This implies that the sequence $\left \{ \frac{\nu_{\mathfrak{a}_{\bullet}}^I (p^e)}{p^e}\right\}$ is an eventually  increasing sequence, and therefore, $\mathcal{C}^I_{\pm}(\aa_\bullet)=  \sup_{e \ge e_0} \frac{\nu_{\mathfrak{a}_{\bullet}}^I (p^e)}{p^e}.$ Hence, the second part follows.
\item We prove this by contradiction. Suppose that for some $e\ge e_0$,  $\frac{\nu_{\aa_{\bullet}}^I(p^e)}{p^e}={\lim_{e \to \infty} \frac{\nu_{\mathfrak{a}_{\bullet}}^I (p^e)}{p^e}}.$ Since $I^{[p^e]}=(I^{[p^e]})^*$ for $e\ge e_0$, by {$(2)$},  $$ \frac{\nu_{\aa_{\bullet}}^I(p^e)}{p^e} \leq \frac{\nu_{\mathfrak{a}_{\bullet}}^I (p^{e+f})}{p^{e+f}} \leq \lim_{e \to \infty} \frac{\nu_{\mathfrak{a}_{\bullet}}^I (p^e)}{p^e} = \frac{\nu_{\aa_{\bullet}}^I(p^e)}{p^e}$$ for all $f$. 
Thus,  $\nu_{\mathfrak{a}_{\bullet}}^I (p^{e+f}) =p^f \nu_{\mathfrak{a}_{\bullet}}^I (p^{e})$ which implies that $\aa_{p^f \nu_{\mathfrak{a}_{\bullet}}^I (p^{e}) +1} \subseteq I^{[p^{e+f}]}$ for all $f$.  Observe that $\aa_{\nu_{\mathfrak{a}_{\bullet}}^I (p^{e})}^{p^f +1} \subseteq \aa_{p^f \nu_{\mathfrak{a}_{\bullet}}^I (p^{e}) +1}  \subseteq I^{[p^{e+f}]}$ for all $f$. Since $I^{[p^e]}$ is tightly closed and $\sqrt \aa_\bullet \cap R^o\neq \emptyset$,  $\aa_{\nu_{\mathfrak{a}_{\bullet}}^I (p^{e})} \subseteq I^{[p^e]}$ which is a contradiction.
\item Suppose that $\aa_r \not\subseteq J^{[p^e]}$. Then, $\aa_r \not\subseteq I^{[p^{e}]}$ as $I^{[p^{e}]} \subseteq J^{[p^e]}$ which implies that $\nu_{\mathfrak{a}_{\bullet}}^J (p^e) \leq \nu_{\mathfrak{a}_{\bullet}}^I (p^{e}).$ Now, the rest follows from definition of limsup and liminf.

\item Note that $\aa_r \not\subseteq I^{[p^{e+e_0}]}$ if and only if $\aa_r \not\subseteq (I^{[p^{e_0}]})^{[p^{e}]}$ as $I^{[p^{e+e_0}]} =(I^{[e_{0}]})^{[p^{e}]}.$ This implies that $\nu_{\mathfrak{a}_{\bullet}}^I (p^{e+e_0}) = \nu_{\mathfrak{a}_{\bullet}}^{I^{[q]}} (p^{e}).$ Now, using this equality in the definition of $\mathcal{C}^I_{\pm}(\aa_\bullet)$, we get the assertion.
\item 
  Suppose that $\aa_r\subseteq I^{[p^e]}$. Then, $\aa_rS\subseteq I^{[p^e]}S=(IS)^{[p^e]}$ which implies that $\nu^I_{\aa_{\bullet}}(p^e)\geq \nu^{IS}_{\aa_{\bullet}S}(p^e)$ for all $e$. Since
 $R\rightarrow S$ is cyclic pure extension, if $\aa_rS\subseteq (IS)^{[p^e]}=I^{[p^e]}S$, then $\aa_r\subseteq I^{[p^e]}$. Therefore, $\nu^I_{\aa_{\bullet}}(p^e)\leq \nu^{IS}_{\aa_{\bullet}S}(p^e)$ for all $e$, and the proof follows.
\end{enumerate} 
\end{proof}

\begin{remark}
{Let $R$ be a Noetherian domain of prime characteristic $p>0$ and $I$ be a non-zero proper ideal of $R$ such that $I^{[p^e]}=(I^{[p^e]})^F$ for $e\gg 0$. We claim that $\mathcal{C}^I(\aa_\bullet)=0$ if and only if $\aa_i =0$ for all $i \ge 1$. If $\mathcal{C}^I(\aa_\bullet)=0,$ then by \Cref{basic-prop} $(1)$, $\nu_{\aa_\bullet}^I(p^e)=0$ for $e\gg 0$, and therefore, $\aa_1 \subseteq \cap_{e \gg 0} I^{[p^e]} =(0)$. Thus, $\aa_i =0$ for all $i \ge 1$. The converse hold trivially.}
 \end{remark}

We now compare $F$-threshold of filtration of ideals when one filtration contains another filtration. 
Let  $\aa_\bullet = \{\aa_i\}_{i \ge 1}$ and $\bb_\bullet = \{\bb_i\}_{i \ge 1}$ be two filtrations in $R$. We write $\aa_\bullet \le \bb_\bullet$ if $\aa_i \subseteq \bb_i$ for all $i \ge 1$. 

\begin{theorem}
\label{fin-fil}  
Let  $\aa_\bullet = \{\aa_i\}_{i \ge 1}$ and $\bb_\bullet = \{\bb_i\}_{i \ge 1}$ be two filtrations in $R$. If $\aa_\bullet \le \bb_\bullet$, then $\mathcal{C}_{\pm}^I(\aa_\bullet) \le \mathcal{C}_{\pm}^I(\bb_\bullet)$. Moreover, if $\mathcal{R}(\bb_\bullet)$ is a finitely generated $\mathcal{R}(\aa_\bullet)$-module, then $\mathcal{C}_{\pm}^I(\aa_\bullet) = \mathcal{C}_{\pm}^I(\bb_\bullet)$.
\end{theorem}
\begin{proof} 
  If  $\mathfrak{a}_r \not\subseteq I^{[p^e]} $, then $\mathfrak{b}_r \not\subseteq I^{[p^e]}$ as $\mathfrak{a}_r \subseteq \mathfrak{b}_r$. Therefore, $\nu_{{\mathfrak{a}_{\bullet}}}^I (p^e) \leq \nu_{\mathfrak{b}_{\bullet}}^I (p^e) $ for all $e$. Thus, by the definition of limsup and liminf,  $\mathcal{C}_{\pm}^I(\aa_\bullet) \le \mathcal{C}_{\pm}^I(\bb_\bullet)$.

Since $\mathcal{R}({\bb_\bullet})$ is finitely generated as an $\mathcal{R}(\aa_\bullet)$-module,  it follows from the proof of \cite[Theorem 3.2]{HKN22} that there exists a positive integer $k$ such that   ${\bb_{n+k}} \subseteq \aa_n$ for all $n$. We claim that $ \nu_{\mathfrak{b}_{\bullet}}^I (p^e) \leq \nu_{\mathfrak{a}_{\bullet}}^I (p^e)+k $ for all $e$. Let $r$  be a non-negative integer such that  $\mathfrak{b}_r \not\subseteq I^{[p^e]} $. If $r >k$, then $\mathfrak{a}_{r-k} \not\subseteq  I^{[p^e]}$ as $\mathfrak{b}_r \subseteq \mathfrak{a}_{r-k}$. Thus, $r-k \leq \nu_{{\mathfrak{a}_{\bullet}}}^I (p^e)$ which further implies that  $ \nu_{\mathfrak{b}_{\bullet}}^I (p^e) \leq \nu_{\mathfrak{a}_{\bullet}}^I (p^e)+k $. Therefore, for all $e$, $ \frac{\nu_{{\mathfrak{b}_{\bullet}}}^I (p^e)}{p^e} \leq \frac{\nu_{{\mathfrak{a}_{\bullet}}}^I (p^e)+k}{p^e}$,  and hence, $\liminf_{e \to \infty}\frac{\nu_{{\mathfrak{b}_{\bullet}}}^I (p^e)}{p^e}\le \liminf_{e \to \infty}\frac{\nu_{{\mathfrak{a}_{\bullet}}}^I (p^e)+k}{p^e}$ and $\limsup_{e \to \infty}\frac{\nu_{{\mathfrak{b}_{\bullet}}}^I (p^e)}{p^e}\le \limsup_{e \to \infty}\frac{\nu_{{\mathfrak{a}_{\bullet}}}^I (p^e)+k}{p^e}.$ Note that  $$\liminf_{e \to \infty}\frac{\nu_{{\mathfrak{a}_{\bullet}}}^I (p^e)+k}{p^e}= \liminf_{e \to \infty}\frac{\nu_{{\mathfrak{a}_{\bullet}}}^I (p^e)}{p^e} \text{ and } \limsup_{e \to \infty}\frac{\nu_{{\mathfrak{a}_{\bullet}}}^I (p^e)+k}{p^e}=\limsup_{e \to \infty}\frac{\nu_{{\mathfrak{a}_{\bullet}}}^I (p^e)}{p^e}.$$ Thus, $\mathcal{C}_{\pm}^I(\aa_\bullet) \ge \mathcal{C}_{\pm}^I(\bb_\bullet)$. Hence, the assertion follows. 
\end{proof}
Let $ \aa$ be an ideal of $R$. A filtration $\aa_\bullet$ in $R$ is said to be \textit{$\aa$-admissible} if there exists a positive integer $k$ such that for all $i$, $\aa^i \subseteq \aa_i \subseteq \aa^{i-k}.$ As an immediate consequence, we obtain the following: 

\begin{proposition}
{Let $\aa, I$ be non-zero proper ideals of $R$ such that $\aa \subseteq \sqrt I$.} Then, for any $\aa$-admissible filtration $\aa_\bullet$, $\mathcal{C}^I(\aa_\bullet) = \mathcal{C}^I(\aa^\bullet)$.
\end{proposition}
\begin{proof}
{Since $\aa_\bullet$ is an $\aa$-admissible filtration, $\aa^{\bullet} \le \aa_\bullet$, and therefore, by \Cref{fin-fil}, $\mathcal{C}_{\pm}^I(\aa^\bullet) \le \mathcal{C}_{\pm}^I(\aa_\bullet)$. Also, there exists a positive integer $k$ such that $\aa_{i} \subseteq \aa^{i-k}$ for all $i\ge k$. Thus, $\nu_{\aa_\bullet}^I(p^e) \le \nu_{\aa^\bullet}^I(p^e)+k$ for all $e.$ Now, applying limsup and liminf to the sequences, we get that $\mathcal{C}_{\pm}^I(\aa_\bullet) \le \mathcal{C}_{\pm}^I(\aa^\bullet)$. Hence, by \cite[Theorem 3.4]{DNP18}, the assertion follows.}
\end{proof}

As another immediate consequence, we obtain the following:
\begin{corollary}
Let $\aa, I$ be non-zero proper ideals of $R$. Then, $\mathcal{C}_{\pm}^I(\overline{\aa^\bullet})=\mathcal{C}_{\pm}^I(( \aa^\bullet)^*).$
\end{corollary}
\begin{proof}
It follows from the well known fact that $( \aa^\bullet)^* \le \overline{\aa^\bullet}$. Next, it follows from \cite[Theorem 13.2.1]{HS06} that  $\R(\overline{\aa^\bullet})$ is finitely generated $\R((\aa^\bullet)^*)$-module. Hence, the assertion follows from Theorem \ref{fin-fil}. 
\end{proof}

Let $\mathfrak{a}_{\bullet}$ be a filtration of ideals in $R$. 
Take $\overline{\mathfrak{a}_{\bullet}} = \{ \overline{\mathfrak{a}_i} \}_{i\geq 0}$ and $({\mathfrak{a}_{\bullet}})^* = \{ ({\mathfrak{a}_i} )^*\}_{i\geq 0}$. Observe that $\overline{\aa_\bullet}$ and $({\mathfrak{a}_{\bullet}})^*$ are filtrations of ideals in $R$ and $\aa_\bullet \le ({\mathfrak{a}_{\bullet}})^* \le  \overline{\aa_\bullet}.$

\begin{corollary}\label{int-fil}
Let $\mathfrak{a}_{\bullet}$ be a filtration in $R$. Suppose that ${\mathcal{R}(\overline{\mathfrak{a}_{\bullet}})}$ is a finitely generated $\mathcal{R}(\mathfrak{a}_{\bullet})$-module. Then, $$\mathcal{C}_{\pm}^I(\aa_\bullet) = \mathcal{C}_{\pm}^I(({\aa_\bullet})^*)=\mathcal{C}_{\pm}^I(\overline{\aa_\bullet}).$$
\end{corollary}
\begin{proof}
 The assertion follows from Theorem \ref{fin-fil}.
\end{proof}
{Since for an analytically unramified local ring $(R,\mm)$, ${\mathcal{R}(\overline{\mathfrak{a}^{\bullet}})}$ is a finitely generated $\mathcal{R}(\mathfrak{a}^{\bullet})$-module for every ideal $\aa$ of $R$, (see \cite[Corollary 9.2.1]{HS06}) we have:}
 \begin{corollary}\label{int-f}
Let $(R,\mathfrak{m})$ be an analytically unramified local ring of prime characteristic $p>0$. Let $\aa, I$ be non-zero proper ideals of $R$ with $\aa \subseteq \sqrt{I}$.  Then,  $\mathcal{C}^I(\aa^\bullet)=\mathcal{C}^I(\overline{\aa^\bullet}).$
\end{corollary}
 
 {With the following example we illustrate
 that \Cref{fin-fil} and \Cref{int-fil} are not necessarily true without the hypothesis that $\mathcal{R}(\bb_\bullet)$ is finitely generated as an  $\mathcal{R}(\mathfrak{a}_{\bullet})$-module.}

\begin{example}
Let $S$ be a regular local ring of characteristic $p>0$ and $\aa,I \subseteq S$ be non-zero proper ideals with $\aa \subseteq \sqrt{I}$. Consider $R=S[[x]]/(x^n)$, for some $n \geq 2$. First note that $R$ being non-reduced, $R$ is not analytically unramified. Let $J=IR$ and $\bb= \aa R$, by Proposition \ref{basic-prop} $(5)$, $\mathcal{C}^J(\bb) =\mathcal{C}^I(\aa).$ Note that $x \in \overline{\bb^i}$ for all $i\ge 0$. We claim that $\mathcal{R}(\overline{\bb^\bullet})$ is not a finitely generated $\mathcal{R}(\bb^\bullet)$-module. For, if $\mathcal{R}(\overline{\bb^\bullet})$ is a finitely generated $\mathcal{R}(\bb^\bullet)$-module,  then there exists a positive integer $k$ such that $\overline{\bb^{n+k}} \subseteq \bb^n$ for all $n$. This implies that $x \in \bb^n$ for all $n$ which is a contradiction. Thus, the claim follows. Next, observe that $\overline{\bb^n} \not \subseteq J^{[p^e]} $ for all $n,e$. Consequently, $\nu_{\overline{\bb^\bullet}}^J(p^e) =\infty$ for all $e \ge 0$. Hence, $\mathcal{C}^J_{\pm}(\overline{\bb^\bullet}) =\infty.$ 
\end{example}

Next,  we  provide a necessary condition for  $\mathcal{C}_+^I(\aa_{\bullet})$ to be a non-negative real number. 

\begin{lemma}\label{exis-nec}
Let  $\mathfrak{a}_{\bullet}$ be a filtration.  Then $\mathcal{C}_+^I(\aa_{\bullet})< \infty$ if and only if there exists a positive integer $M$ such that $\aa_{Mp^e} \subseteq I^{[p^e]}$ for all $e \gg 0$. Moreover, in this case, we have $\sqrt{\aa_\bullet} \subseteq \sqrt{I}$. 
\end{lemma}
\begin{proof}
 Note that $\mathcal{C}_+^I(\aa_{\bullet})< \infty$ if and only if there exists  an integer $M>0$ such that  $\sup_{f \ge e} \frac{\nu_{\mathfrak{a}_{\bullet}}^I (p^f)}{p^f} <M$ for all $e \gg 0$, which happens if and only if $\nu_{\mathfrak{a}_{\bullet}}^I (p^e) <Mp^e$ for all $e\gg0$. Thus, $\mathcal{C}_+^I(\aa_{\bullet})< \infty$ if and only if  $\aa_{Mp^e} \subseteq I^{[p^e]}$ for all $e\gg 0$. Since  $\sqrt{\aa_{Mp^e}} \subseteq \sqrt{I}$ for all $e \gg 0$ and  $\sqrt{\aa_\bullet} = \sqrt{\aa_{Mp^e}}$, the second assertion follows.
\end{proof}

It is natural to ask whether $\sqrt{\aa_\bullet} \subseteq \sqrt{I}$ is a sufficient condition for $\mathcal{C}_+^I(\aa_{\bullet})$ to be a non-negative real number.   Here, we provide an  example which shows that the condition $\sqrt{\aa_\bullet} \subseteq \sqrt{I}$ is not sufficient for $\mathcal{C}_+^I(\aa_{\bullet})$ to be a non-negative real number. 
\begin{example}
Let $S$ be a  Noetherian domain of prime characteristic $p$ and $I$ be a non-zero proper ideal of $S$. Take $R=S[x]$ and $\aa_n=Ix^n$ for every $n$. Clearly, $\aa_\bullet$ is a filtration and $\sqrt {\aa_\bullet} = \sqrt {Ix} \subseteq \sqrt{IR}.$ However, for every $ e, r \ge 1$, $\aa_r \not\subseteq (IR)^{[p^e]}.$ Therefore,  $\mathcal{C}_{\pm}^{IR}(\aa_{\bullet})=\infty.$
\end{example}

In the following, we provide a sufficient condition on a filtration of ideals so that $\mathcal{C}^I_+(\aa_\bullet) < \infty.$  

\begin{theorem}\label{suff}
 Let $\aa_{\bullet}$ be a filtration such that for some $N \in \mathbb{N}$, ${\aa_{Ns}} \subseteq {I}^s$ for all $s \ge 0$. Then $ 0 \le \mathcal{C}^I_{\pm}(\aa_\bullet) \le N\mu(I)$, where $\mu(I)$ denotes the number of elements in a minimal generating set of $I$. Particularly, if $R$ is  F-pure, then $\mathcal{C}^I(\aa_\bullet)$ exists  and $$ 0 \le \mathcal{C}^I(\aa_\bullet)= \sup_{e \ge 0} \frac{\nu_{\aa_\bullet}^I(p^e)}{p^e}\le N\mu(I).$$
\end{theorem}

\begin{proof}
First, we claim that $\nu_{\aa_\bullet}^I(p^e) \le N\mu(I)(p^e-1)+N-1$ for all $e \ge 0$. For every $e \ge 0$, if $s \ge N \mu(I) (p^e-1) +N$, then 
\begin{align*} \aa_s & \subseteq \aa_{N \mu(I) (p^e-1) +N}  \subseteq I^{\mu(I) (p^e-1) +1}  \subseteq I^{[p^e]}, \end{align*} where the last containment follows from {Pigeonhole Principle}  (\cite[Lemma 2.4(1)]{HH02}).  Therefore, $\nu_{\aa_\bullet}^I(p^e) \le N\mu(I)(p^e-1)+N-1$  which implies that $$ \frac{\nu_{\aa_\bullet}^I(p^e)}{p^e} \le \frac{N\mu(I)(p^e-1)+N-1}{p^e}.$$ Now, applying $\limsup $ and $\liminf$ we get that $0 \le \mathcal{C}^I_{\pm}(\aa_\bullet)  \le N \mu(I).$  

Now, if $R$ is an $F$-pure ring, $I^{[p^e]} = (I^{[p^e]})^F$ for all $ e \ge 0$. By Proposition \ref{basic-prop} {$(1)$}, $$ {\mathcal{C}_{\pm}^I(\aa_\bullet)=\lim_{e \to \infty}\frac{\nu_{\aa_\bullet}^I(p^e)}{p^e}= \sup_{e \ge 0} \frac{\nu_{\aa_\bullet}^I(p^e)}{p^e}.}$$ Since $\left\{\frac{\nu_{\aa_\bullet}^I(p^e)}{p^e}\right\}$ is a bounded sequence, $\mathcal{C}^I(\aa_\bullet)$ exists, and hence, the assertion follows. 
\end{proof}
We now illustrate the above result with the following example.
\begin{example}\label{non-noe-fil}
Let $R=\mathbb{K}[x_1,\ldots,x_n]$ and $\mathfrak{m}=(x_1,\ldots,x_n)$, where   $\mathbb{K}$ is a field of prime characteristic {$p>0$}. Let $\aa$ and $I$ be non-zero proper homogeneous ideals of $R$ such that $\aa \subseteq I$ and $k$ be a positive integer. Take $\aa_i=\aa^{\lceil \frac{i}{k} \rceil +1}$ for $i\ge 1$. It is easy to see that $\aa_\bullet$ is a filtration of ideals in $R$. Note that $\aa_{ki} =\aa^{i+1} \subseteq I^i $ for all $i$. Thus, by Theorem \ref{suff}, $\mathcal{C}^{I}(\aa_\bullet) $ exists. One can also observe that $\R(\aa_\bullet)$ is not finitely generated {$R$-algebra}, see \cite{Ha-Nguyen}.
\end{example}

Next we prove that if $\aa_\bullet$ is a Noetherian filtration, then $\sqrt{\aa_\bullet} \subseteq \sqrt{I}$ is sufficient for $\mathcal{C}_{\pm}^I(\aa_{\bullet})$ to be non-negative real number as well as for the existence of  $\mathcal{C}^I(\aa_{\bullet})$.

\begin{theorem} \label{noe-fil}  
If $\aa_{\bullet}$ is a Noetherian filtration with $\sqrt{\aa_\bullet} \subseteq \sqrt{I}$, then $ 0 \le \mathcal{C}^I(\aa_\bullet)< \infty$, and it is equal to $r  \mathcal{C}^I(\aa_r^\bullet)$ for some $r \in \mathbb{N}$. Moreover, if $R$ is regular $F$-finite ring {and essentially of finite type over a field}, then $\mathcal{C}^I(\aa_\bullet)$ is rational.
\end{theorem}
\begin{proof}
Since $\aa_\bullet$ is a Noetherian filtration, by \cite[Proposition 2.1]{P88}, there exists a positive integer $r$ such that $\R^{[r]}(\aa_\bullet)$ is standard graded $R$-algebra, i.e., $\aa_{nr}=\aa_r^n$ for all $n \ge 0$. Note that $\aa_r \subseteq \sqrt{\aa_\bullet} \subseteq \sqrt{I}$. So, there exists a positive integer $N$ such that  $\aa_r^N \subseteq I.$ This implies that $\aa_{rNs} = \aa_r^{Ns} \subseteq I^s$ for all $s \ge 1$.  By Theorem \ref{suff}, $0 \le \mathcal{C}^I_{\pm}(\aa_\bullet)<\infty$. Consequently, $\nu_{\mathfrak{a}_{\bullet}}^I (p^e) < \infty$ for all $e \ge 0$.  Next, we claim that for all $e \ge 0$, $$r  \nu_{{\mathfrak{a}_r^{\bullet}}}^I (p^e) \leq \nu_{\mathfrak{a}_{\bullet}}^I (p^e) \leq r (\nu_{{\mathfrak{a}_r^{\bullet}} }^I (p^e)+1)-1.$$ Let $e\ge 0$ be any.  Suppose that  $\aa_r^n \not\subseteq I^{[p^e]}$ for some $n \ge 0$, then  $\aa_{rn} =\aa_r^n \not\subseteq I^{[p^e]}$ which implies that $r  \nu_{{\mathfrak{a}_r^{\bullet}}}^I (p^e) \leq \nu_{\mathfrak{a}_{\bullet}}^I (p^e)$.   For $s=\nu_{{\mathfrak{a}_r^{\bullet}}}^I (p^e) +1$, we have $\aa_r^s \subseteq I^{[p^e]}$ which implies that $\aa_{rs} \subseteq I^{[p^e]}$. Thus,  $\nu_{\mathfrak{a}_{\bullet}}^I (p^e) \leq r  (\nu_{{\mathfrak{a}_r^{\bullet}}}^I (p^e) +1)-1$, and hence, the claim follows.  Consequently, we get 
 $$r \frac{\nu_{{\mathfrak{a}_r^{\bullet}}}^I (p^e)}{p^e} \leq \frac{\nu_{\mathfrak{a}_{\bullet}}^I (p^e)}{p^e} \leq r  \frac{\nu_{{\mathfrak{a}_r^{\bullet}} }^I (p^e)}{p^e} +\frac{r-1}{p^e}, $$ for $e \geq 0$. By \cite[Theorem 3.4]{DNP18}, we know that $\mathcal{C}^I(\aa_r^\bullet)$ exists. Now, by applying Sandwich theorem in the above inequality,  we get $$\lim_{e \to \infty} \frac{\nu_{{\mathfrak{a}_{\bullet}}}^I (p^e)}{p^e} =r  \lim_{e \to \infty} \frac{\nu_{{\mathfrak{a}_r^{\bullet}}}^I (p^e)}{p^e}=r  \mathcal{C}^I(\aa_r^\bullet).$$ Hence, $\mathcal{C}^I(\aa_\bullet)$ exists and it is equal to $r  \mathcal{C}^I(\aa_r^\bullet)$. The rationality of $\mathcal{C}^I(\aa_\bullet)$ follows from \cite[{Theorem 3.1}]{MMK08}. 
\end{proof}
\begin{example}
Let $\aa,\bb$ be non-zero proper ideals of $R$ such that $\aa^2 \subsetneq \bb \subsetneq \aa$ and  $\bb \subseteq \sqrt{I}.$ By \cite[Theorem 3.4]{DNP18}, $\C^I(\bb^{\bullet})$ exists. Now, take $\aa_0=R, \aa_1=\aa, \aa_2=\bb$, $\aa_{2i+1}=\aa \bb^i$ and $\aa_{2i}=\bb^i$. One can verify that $\aa_\bullet$ is a Noetherian filtration of ideals in $R$ with $\aa_{2i}=\bb^i=\aa_2^i$ for all $i$.  Therefore, by Theorem \ref{noe-fil}, $\mathcal{C}^I(\aa_\bullet)$ exists and it is equal to $2 \mathcal{C}^I(\bb^\bullet)$.
\end{example}

{The Rees algebra of a filtration of ideals in general is not necessarily Noetherian, see Example \ref{non-noe-fil} or even when the filtration    is of symbolic powers of an ideal in $R$ (cf. \cite{C91, H82, R85}).} {So, next, we prove that symbolic $F$-threshold of a non-zero ideal $\aa$ with respect to $I$ exists, i.e., $\mathcal{C}^I(\aa^{(\bullet)})$ is a non-negative real number  if $\aa \subseteq \sqrt I.$ Recall that the \emph{big-height} of an ideal $I$, denoted by $\text{big-height}(I)$, is $\max\{\hgt( \mathfrak{p}):\mathfrak{p}\in \Ass(I)\}$, where $\Ass(I)$ denotes the set of associated primes of $I$}.
\begin{theorem}\label{exist}
Assume that $R$ is a regular ring. Let $\aa$ and  {$I$} be a non-zero proper ideals in $R$ {with} $\aa \subseteq \sqrt{I}$.  Then $\mathcal{C}^I(\aa^{(\bullet)})$ exists, where $\aa^{(\bullet)}= \{ \aa^{(i)}\}_{i \ge 0}$ is the symbolic power filtration of $\aa$.
\end{theorem}
\begin{proof}
First, note that $\aa^{(\bullet)} \le \sqrt{\aa}^{(\bullet)}$.  Since $R$ is a regular ring, $I^{[p^e]} = (I^{[p^e]})^F$ for all $ e \ge 0$.
By Proposition \ref{basic-prop} and Theorem \ref{fin-fil}, $${ \mathcal{C}_{\pm}^I(\aa^{(\bullet)})= \sup_{e \ge 0} \frac{\nu_{\aa^{(\bullet)}}^I(p^e)}{p^e} \le \sup_{e \ge 0} \frac{\nu_{\sqrt{\aa}^{(\bullet)}}^I(p^e)}{p^e}= \mathcal{C}_{\pm}^I(\sqrt{\aa}^{(\bullet)}).}$$ So, it is enough to prove the assertion for radical ideals. So without loss of generality we may assume that $\aa$ is a radical ideal. By \cite[Theorem 2.6]{HH02}, $\aa^{(Hs)} \subseteq \aa^s$ for all $s$, where $H=\text{big-height}(\aa)$.  Since $\aa \subseteq \sqrt I$, there exists a positive integer $t$ such that $\aa^t \subseteq I$. Now, for $N=Ht$, $\aa^{(Ns)} \subset \aa^{ts} \subseteq I^s$ for all $s.$ Thus, by Theorem \ref{suff}, $\mathcal{C}^I(\aa^{(\bullet)})$ exists. Hence, the assertion follows. 
\end{proof}

{Since $\aa^{\bullet}\leq \aa^{(\bullet)}$, we always have $\mathcal{C}^I(\aa^{\bullet})\leq \mathcal{C}^I(\aa^{(\bullet)})$.
In the following example, we illustrate that in general, $F$-threshold of symbolic filtration need not be equal to $F$-threshold of ordinary power filtration. 
\begin{example}\label{sym-ord-exm}
Let $R=\mathbb{K}[x_1,\ldots,x_n]$ and $\mathfrak{m}=(x_1,\ldots,x_n)$, where   $\mathbb{K}$ is a field of prime characteristic $p>0$.  Let $\aa=\bigcap_{1 \le i<j \le n}(x_i,x_j) =\left(\frac{x_1 \cdots x_n}{x_i} : 1 \le i \le n\right)$.  Observe that $\aa^{(n)}=\bigcap_{1 \le i<j \le n}(x_i,x_j)^n$ for all $n.$ We claim that $\aa^{(r)} \subseteq \mathfrak{m}^{[p^e]}$ if and only if $r \ge 2(p^e-1)+1.$ If $r \ge 2(p^e-1)+1$, then $(x_i,x_j)^{r} \subseteq (x_i,x_j)^{2(p^e-1)+1} \subseteq (x_i,x_j)^{[p^e]}$  {where the last inclusion follows from Pigeonhole Principle} (cf. \cite[Lemma 2.4 (1)]{HH02}).  Consequently, $\aa^{(r)} =  \bigcap_{1 \le i<j \le n}(x_i,x_j)^r \subseteq \mathfrak{m}^{[p^e]}.$ One can note that $(x_1 \cdots x_n)^{p^e-1} \in \bigcap_{1 \le i <j \le n}(x_i,x_j)^{2(p^e-1)}=\aa^{(2(p^e-1))}$ and that $(x_1 \cdots x_n)^{p^e-1} \not\in \mathfrak{m}^{[p^e]}.$ Thus, the claim follows. Consequently, we get  $\nu_{\mathfrak{a}^{(\bullet)}}^{\mathfrak{m}} (p^e)= 2(p^e-1)$ for all $e$.
Hence,
$$\mathcal{C}^{\mathfrak{m}}(\aa^{(\bullet)})=\lim_{e \to \infty} \frac{\nu_{{\mathfrak{a}^{(\bullet)}}}^{\mm} (p^e)}{p^e}=\lim_{e \to \infty} \frac{2(p^e-1)}{p^e}=2.$$ 

 Next, we claim that $\C^{\mm}(\aa^\bullet)=\frac{n}{n-1}.$ First, note that $(x_1 \cdots x_n)^{n-1} \in \aa^n$, and therefore, for all $e$, $\left((x_1\cdots x_n)^{n-1} \right)^{\lfloor \frac{p^e-1}{n-1} \rfloor} \in \aa^{n\lfloor \frac{p^e-1}{n-1} \rfloor}$.  Since $p^e -1 \ge (n-1) \lfloor \frac{p^e-1}{n-1} \rfloor$, we get that  for all $e$, $n\lfloor \frac{p^e-1}{n-1} \rfloor \le \nu_{\aa^\bullet}^{\mm}(p^e).$ Next, we claim that if $r\ge n\left( \frac{p^e-1}{n-1}\right) +1$, then $\aa^r \subseteq \mm^{[p^e]}$. Note that for any non-negative integers  $\alpha_1, \ldots, \alpha_n$ with $\alpha_1+\cdots+\alpha_n=r$, $\prod_{i=1}^n \left(\frac{x_1 \cdots x_n}{x_i}\right)^{\alpha_i}=\prod_{i=1}^n x_i^{r-\alpha_i} .$  If $r-\alpha_i \le p^e-1$ for all $i$, then $nr-r \le n(p^e-1)$ which is a contradiction to the fact that  $r\ge n\left( \frac{p^e-1}{n-1}\right) +1$. Thus, the second claim follows, and hence,  $n\lfloor \frac{p^e-1}{n-1} \rfloor \le \nu_{\aa^\bullet}^{\mm}(p^e) \le n\left( \frac{p^e-1}{n-1}\right) +1$ for all $e$.  Hence, it follows immediately from Sandwich Theorem of Limits that  
$$\mathcal{C}^{\mathfrak{m}}(\aa^{\bullet})=\lim_{e \to \infty} \frac{\nu_{{\mathfrak{a}^{\bullet}}}^{\mm} (p^e)}{p^e}=\frac{n}{n-1}.$$ \end{example}

%%%%%%%%%%%%%%%%%%%%%%%%%%%%%%%%%%%%%%%%%%%%%

\section{$(I,p)$-admissible filtrations and their $F$-thresholds}\label{sec-sym}
{In this section, we introduce a class of filtrations which we call as $(I,p)$-admissible filtrations and study their $F$-thresholds. We show that the ordinary power filtration, the symbolic power filtration and  the integral closure power filtration of an ideal are standard examples of $(I,p)$-admissible filtrations. Moreover, we show that $(I,p)$-admissible filtration are closed under products, intersections (for regular rings) and binomial sums. We begin with the definition of $(I,p)$-admissible filtration.} 
\begin{definition}
 Let $R$ be a ring of prime characteristic $p>0$. Let $\aa_\bullet$ be a filtration and $I$ be an ideal in $R$.   We  say that $\mathfrak{a}_{\bullet}$ is an \emph{$(I,p)$-admissible  filtration}, if 
   \begin{enumerate}
    \item  for some $k$, $\mathfrak{a}_{k}\subset I$, and 
    \item  there exist integers $h>0$  and  $c$ such that $$\mathfrak{a}_{(h+m)p^e+c}\subseteq \mathfrak{a}_{m+1}^{[p^e]} \text{ for all } m, e\geq 0.$$
   \end{enumerate}
\end{definition}
Next, we see that for a non-zero ideal $\aa$ with $ \aa \subset \sqrt I$, the standard filtrations associated with $\aa$ are $(I,p)$-admissible  filtrations when $R$ or $\aa$ satisfy some nice properties. 
\begin{example} Let $\aa, I$ be proper  non-zero ideals such that $\aa \subseteq \sqrt I$. Let $\aa$ be minimally generated by $f_1,\ldots,f_h$ and let $H$ be the big height of $\aa$.
\par (a)  It follows from \cite[Lemma 2.4]{HH02}  that $\aa^{(h+m)p^e}\subseteq (\aa^{m+1})^{[p^e]}$ for all $m, e \ge 0.$ Therefore, $\mathfrak{a}^{\bullet}$ is  an $(I,p)$-admissible filtration.
\par (b) {Assume that $R$ is a regular ring and $\aa$ is a radical ideal.   By \cite[Theorem 2.6]{HH02}, $\aa^{(Hs)} \subseteq \aa^s$ for all $s$. Since $\aa \subseteq \sqrt I$, there exists a positive integer $k$ such that $\aa^k \subseteq I$. Therefore, $\aa^{(Hk)} \subseteq I$.  Now, by \cite[Lemma 2.6]{GH17}, for all $m, e\geq 0$,
$$\aa^{((H+m)p^e-H+1)}\subseteq \left(\aa^{(m+1)}\right)^{[p^e]}.$$
Thus $\aa^{(\bullet)}$ is an $(I,p)$-admissible filtration.}
%\par (c) Assume that $R$ is a local ring with the unique maximal ideal $\mm$ and $\aa$ is a radical ideal such that for each $k$, $\aa^{(k)}$ has finite projective dimension. By \cite{HH02},  we know that for all  $k, e\geq 0$,
%$$\aa^{(Hp^e+kp^e-h+1)}\subseteq {\aa^{(k+1)}}^{[p^e]}.$$
%Thus, in this case $\aa^{(\bullet)}$  is  an $(\mathfrak{m},p)$-admissible filtration.

\par (d) Assume that $\aa$ satisfies Brian\c{c}on-Skoda condition, i.e., there exists a positive integer $k$ such that $\overline{\aa^{k+m}}\subseteq \aa^{m}$ for all $m \ge 0$ (for example every ideal in reduced $F$-finite ring of positive characteristic $p$ satisfy Brian\c{c}on-Skoda condition, see \cite[Theorem 13.4.8]{HS06}). Then, for all $m,e\ge 0$, using \cite[Lemma 2.4]{HH02}, we get $\overline{\aa^{(k+h+m)p^e}}\subseteq \overline{\aa^{(h+m)p^e+k}}\subseteq \aa^{(h+m)p^e}\subseteq (\aa^{m+1})^{[p^e]} \subseteq (\overline{\aa^{m+1}})^{[p^e]}$. Therefore, $\overline{\mathfrak{a}^{\bullet}}$ is  an $(I,p)$-admissible filtration.
\par (e) Let $\aa_{\bullet}$ be a filtration of ideals in $R$ such that $I^\bullet \le \aa_\bullet$ and there exists a positive integer $N $ such that ${\aa_{(N+s)p^e}} \subseteq ({I}^{s+1})^{[p^e]}$ for all $s,e \ge 0$.  Then, $\aa_\bullet$ is an $(I,p)$-admissible filtration. 
 \end{example}

Next, we show that one can construct new $(I,p)$-admissible filtrations using given $(I,p)$-admissible filtrations. One can see that if $\aa_\bullet$ and $\bb_\bullet$ are  two filtrations then 
\begin{enumerate}
\item $\cc_\bullet= \{ \aa_n \bb_n\}_{n\geq 0}$ is a filtration;
\item $\dd_\bullet= \{ \aa_n\cap \bb_n\}_{n\geq 0}$ is a filtration;
\item $\ee_\bullet= \{ \sum_{i=0}^n\aa_i\bb_{n-i}\}_{n\geq 0}$ is a filtration.
\end{enumerate}
\begin{proposition}\label{filwith} Let $\aa_\bullet$ and $\bb_\bullet$ be  two $(I,p)$-admissible filtrations. Then, 
\begin{enumerate}
\item $\cc_\bullet= \{ \aa_n \bb_n\}_{n\geq 0}$ is an $(I,p)$-admissible filtration;
\item if $R$ is regular, then $\dd_\bullet= \{ \aa_n\cap \bb_n\}_{n\geq 0}$ is an $(I,p)$-admissible filtration;
\item $\ee_\bullet= \{ \sum_{i=0}^n\aa_i\bb_{n-i}\}_{n\geq 0}$ is an $(I,p)$-admissible filtration.
\end{enumerate}
\end{proposition}
\begin{proof}
Since $\aa_\bullet$ and $\bb_\bullet$ are $(I,p)$-admissible filtrations, there exist  integers $k,k',h,h', c,c'$ such that $\mathfrak{a}_{k}\subset I$, $\bb_{k'} \subseteq I$, and for all $m, e\geq 0$, {$$\mathfrak{a}_{(h+m)p^e+c}\subseteq \mathfrak{a}_{m+1}^{[p^e]} \text{ and } \mathfrak{b}_{(h'+m)p^e+c'}\subseteq \mathfrak{b}_{m+1}^{[p^e]}.$$}
\par (1-2) Note that $\cc_k \subseteq \dd_k \subseteq \aa_k \subseteq I$. Let $l=\max\{h,h'\}, n=\max\{c,c'\}.$ Then, for all $m,e \ge 0$, $$\mathfrak{c}_{(l+m)p^e+n}= \mathfrak{a}_{(l+m)p^e+n}\mathfrak{b }_{(l+m)p^e+n}\subseteq \mathfrak{a}_{(h+m)p^e+c}\mathfrak{b}_{(h'+m)p^e+c'}\subseteq \mathfrak{a}_{m+1}^{[p^e]}\mathfrak{b}_{m+1}^{[p^e]}=\cc_{m+1}^{[p^e]}$$ and $$\mathfrak{d}_{(l+m)p^e+n}= \mathfrak{a}_{(l+m)p^e+n}\cap \mathfrak{b }_{(l+m)p^e+n}\subseteq \mathfrak{a}_{(h+m)p^e+c}\cap \mathfrak{b}_{(h'+m)p^e+c'}\subseteq \mathfrak{a}_{m+1}^{[p^e]}\cap \mathfrak{b}_{m+1}^{[p^e]}=\dd_{m+1}^{[p^e]},$$  where the last equality follows from the flatness of  the Frobenius map on $R$. Thus,  $\cc_\bullet$ and $\dd_\bullet$ are $(I,p)$-admissible filtrations.
\par (3) Let $k''= k+k'$. Then, for every $ 0 \leq i \leq k''$, either $i \ge k$ or $k''-i \ge k'$. This implies that either $\aa_i \subseteq \aa_k \subseteq I$ or $\bb_{k''-i} \subseteq \bb_{k'} \subseteq I$ which further implies that $\dd_{k''} \subseteq I.$ Next, let $l=h+h', n=c+c' $ and  $m,e \geq 0$ be any. Consider, $$\dd_{(l+m)p^e+n} =\sum_{i=0}^{(l+m)p^e+n} \aa_i \bb_{(l+m)p^e+n-i} .$$ Suppose that $0 \leq i <hp^e+c$. Then, {$(l+m)p^e+n-i \geq (h'+m)p^e+c'$ which implies that $\aa_i \bb_{(l+m)p^e+n-i} \subseteq \bb_{(h'+m)p^e+c'} \subseteq \bb_{m+1}^{[p^e]}.$} Now, for  $0 \le j < m$,  if $ (h+j)p^e+c \leq i < (h+j+1)p^e+c$, then $(l+m)p^e+n-i \geq (h'+m-j-1)p^e+c'$. Therefore, $\aa_i \subseteq \aa_{(h+j)p^e+c} \subseteq \aa_{j+1}^{[p^e]}$ and $\bb_{(l+m)p^e+n-i} \subseteq  \bb_{(h'+m-j-1)p^e+c'} \subseteq \bb_{m-j}^{[p^e]}$ which implies that $\aa_i \bb_{(l+m)p^e+n-i}\subseteq \aa_{j+1}^{[p^e]}\bb_{m-j}^{[p^e]}=(\aa_{j+1}\bb_{m-j})^{[p^e]}.$ Now, for $i\ge (h+m)p^e+c$, we have $\aa_i \subseteq \aa_{(h+m)p^e+c} \subseteq \aa_{m+1}^{[p^e]}.$ Therefore, $$\dd_{(l+m)p^e+n} =\sum_{i=0}^{(l+m)p^e+n} \aa_i \bb_{(l+m)p^e+n-i} \subseteq \sum_{j=0}^{m+1} (\aa_j\bb_{m+1-j})^{[p^e]}=\dd_{m+1}^{[p^e]} .$$ 
Thus,  $\dd_\bullet$ is an $(I,p)$-admissible filtration.
\end{proof}

 Now, we prove that the $F$-threshold of an $(I,p)$-admissible filtration exists. The strategy of the proof is the same as the proof of \cite[Theorem 3.4]{DNP18}.
  \begin{theorem}
  \label{F-threshold-*-filtration}
   Let $\mathfrak{a}_{\bullet}$ be an $(I,p)$-admissible filtration. Then
   $\C^{I}(\aa_\bullet)=\lim\limits_{e\to\infty}\frac{\nu^I_{\mathfrak{a}_{\bullet}}(p^e)}{p^e}$ exists. 
  \end{theorem}
  \begin{proof}
  Since $\aa_\bullet$ is an $(I,p)$-admissible filtration, {there exist integers $k, c$ and $h>0$} such that $\aa_k \subseteq I$ and $\mathfrak{a}_{(h+m)p^e+c}\subseteq \mathfrak{a}_{m+1}^{[p^e]} \text{ for all } m, e\geq 0.$ Therefore, $\mathfrak{a}_{(h+k-1)p^e+c}\subseteq \mathfrak{a}_{k}^{[p^e]} \subseteq I^{[p^e]} $ which implies that {$0\leq {\nu^I_{\mathfrak{a}_{\bullet}}(p^e)} \le (h+k-1)p^e+c-1 \text{ for all }  e\geq 0.$} Thus for all large $q=p^e\gg 0$, $0\leq \frac{\nu^I_{\mathfrak{a}_{\bullet}}(p^e)}{p^e}\leq h+k$, i.e. $\{\frac{\nu^I_{\mathfrak{a}_{\bullet}}(p^{e})}{p^{e}}\}$ is an eventually bounded sequence. 
  Now, it suffices to prove that
$$\limsup\limits_{e\to\infty}\frac{\nu^I_{\mathfrak{a}_{\bullet}}(p^{e})}{p^{e}}\leq \liminf\limits_{e\to\infty}\frac{\nu^I_{\mathfrak{a}_{\bullet}}(p^{e})}{p^{e}}.$$ Let $e, e' \ge 0$ and $m=\nu^I_{\mathfrak{a}_{\bullet}}(p^{e})$. Then
 $$\mathfrak{a}_{(h+m)p^{e'}+c}\subseteq \mathfrak{a}_{m+1}^{[p^{e'}]}\subseteq {(I^{[p^{e}]})}^{[p^{e'}]}=I^{[p^{e+e'}]}.$$
 Therefore, $\nu^I_{\mathfrak{a}_{\bullet}}(p^{e+e'})< (h+m)p^{e'}+c$ which implies that
 $$\frac{\nu^I_{\mathfrak{a}_{\bullet}}(p^{e+e'})}{p^{e+e'}}<\frac{h}{p^{e}}+ \frac{\nu^I_{\mathfrak{a}_{\bullet}}(p^{e})}{p^{e}}+\frac{c}{p^{e+e'}}.$$
 Now, for each $e\ge 1$,  we get 
 $$\limsup\limits_{e'\to\infty}\frac{\nu^I_{\mathfrak{a}_{\bullet}}(p^{e+e'})}{p^{e+e'}}\leq \frac{h}{p^{e}}+ \frac{\nu^I_{\mathfrak{a}_{\bullet}}(p^{e})}{p^{e}}.$$ Therefore, 
 $$\C_+^I(\aa_\bullet) =\limsup\limits_{e'\to\infty}\frac{\nu^I_{\mathfrak{a}_{\bullet}}(p^{e+e'})}{p^{e+e'}}\leq \frac{h}{p^{e}}+ \frac{\nu^I_{\mathfrak{a}_{\bullet}}(p^{e})}{p^{e}}.$$ Now, applying $\liminf$ as $e \to \infty$, we get 
 $$\C_+^I(\aa_\bullet)\le \liminf\limits_{e\to \infty} \left(\frac{h}{p^{e}}+ \frac{\nu^I_{\mathfrak{a}_{\bullet}}(p^{e})}{p^{e}}\right)=\liminf\limits_{e\to \infty} \frac{\nu^I_{\mathfrak{a}_{\bullet}}(p^{e})}{p^{e}}$$ which implies that
 $\C_+^I(\aa_\bullet) \le \C_-^I(\aa_\bullet).$ Hence, $\C^I(\aa_\bullet)$ exists.
  \end{proof}
  %We state an immediate consequence of above theorem.
  %\begin{corollary}
 %Let $(R,\mathfrak{m})$ be a local ring and $I$ be a radical ideal of $R$ such that for each $k$, $I^{(k)}$ has finite projective dimension. Then $\mathcal{C}^{\mathfrak{m}}(\aa^{(\bullet)})$ exists and  it is bounded above by the big-height of $I$.
%\end{corollary}
 % \begin{proof}
  % Since $\aa^{(\bullet)}$  is an $(\mathfrak{m},p)$-filtration, the existence of $\mathcal{C}^{\mathfrak{m}}(\mathfrak{a}_{\bullet})$ follows from above theorem. Let $h=$ big-height of $I$, then for all large $q=p^e\gg 0$, $0\leq \frac{\nu^I_{\mathfrak{a}_{\bullet}}(p^e)}{p^e}\leq h$, hence the bound attains. 
  %\end{proof}}

Now, we study $F$-thresholds of filtrations $\cc_\bullet, \dd_\bullet$ and $\ee_\bullet$ in terms of $F$-thresholds of filtrations $\aa_\bullet$ and $\bb_\bullet$.

\begin{proposition}\label{new_fil}
Let $\aa_\bullet$ and $\bb_\bullet$ be two filtrations. Let $\cc_\bullet= \{ \aa_n \bb_n\}_{n\geq 0}$, $\dd_\bullet =  \{ \aa_n \cap \bb_n\}_{n\geq 0}$ and $\ee_\bullet= \{ \sum_{i=0}^n\aa_i\bb_{n-i}\}_{n\geq 0}$. Then, 
$$\mathcal{C}_{\pm}^I(\cc_\bullet) \le \mathcal{C}_{\pm}^I(\dd_\bullet) \le \mathcal{C}_{\pm}^I(\aa_\bullet), \mathcal{C}_{\pm}^I(\bb_\bullet) \le \mathcal{C}_{\pm}^I(\ee_\bullet) \text{ and }  \mathcal{C}_{+}^I(\ee_\bullet) \le \mathcal{C}_{+}^I(\aa_\bullet)+ \mathcal{C}_{+}^I(\bb_\bullet) .$$ \end{proposition}
\begin{proof}
First, note that  $\cc_n \subseteq \dd_n \subseteq \aa_n, \bb_n \subseteq \ee_n$ for all $n$.  Therefore, $\cc_\bullet \le \dd_\bullet \le \aa_\bullet, \bb_\bullet \le \ee_\bullet$. Now, applying Theorem \ref{fin-fil}, we get $\mathcal{C}_{\pm}^I(\cc_\bullet) \le \mathcal{C}_{\pm}^I(\dd_\bullet) \le \mathcal{C}_{\pm}^I(\aa_\bullet), \mathcal{C}_{\pm}^I(\bb_\bullet) \le \mathcal{C}_{\pm}^I(\ee_\bullet)$. So, it remains to prove that  $\mathcal{C}_{+}^I(\ee_\bullet) \le \mathcal{C}_{+}^I(\aa_\bullet)+ \mathcal{C}_{+}^I(\bb_\bullet) .$ 
We claim that $\nu_{\ee_\bullet}^I(p^e) \le \nu_{\aa_\bullet}^I(p^e)+ \nu_{\bb_\bullet}^I(p^e)$ for all $e \ge 0$.   If $\nu_{\aa_\bullet}^I(p^e)+ \nu_{\bb_\bullet}^I(p^e) = \infty$, then vacuously the claim follows. So, we  assume that $\nu_{\aa_\bullet}^I(p^e)+ \nu_{\bb_\bullet}^I(p^e) < \infty.$ Let $k > \nu_{\aa_\bullet}^I(p^e)+ \nu_{\bb_\bullet}^I(p^e).$ Then for every $0 \le i \le k$, either $i>\nu_{\aa_\bullet}^I(p^e)$ or $k-i > \nu_{\bb_\bullet}^I(p^e).$ Therefore, either $\aa_i \subseteq I^{[p^e]} $ or $\bb_{k-i} \subseteq I^{[p^e]}$ which implies that $\aa_i\bb_{k-i} \subseteq I^{[p^e]}$ for every $0 \le i \le k$. Consequently,  $\ee_k \subseteq I^{[p^e]}.$ Thus, the claim follows.  Using claim, we get 
$$\frac{\nu_{{\mathfrak{e}_{\bullet}}}^I (p^e)}{p^e} \leq \frac{\nu_{{\mathfrak{a}_{\bullet}}}^I (p^e)}{p^e} + \frac{\nu_{{\mathfrak{b}_{\bullet}}}^I (p^e)}{p^e}$$ for all $e \geq 0$. Therefore, $$\limsup_{e \to \infty} \frac{\nu_{{\mathfrak{e}_{\bullet}}}^I (p^e)}{p^e} \leq \limsup_{e \to \infty}\left( \frac{\nu_{{\mathfrak{a}_{\bullet}}}^I (p^e)}{p^e} + \frac{\nu_{{\mathfrak{b}_{\bullet}}}^I (p^e)}{p^e}\right) \le \limsup_{e \to \infty} \frac{\nu_{{\mathfrak{a}_{\bullet}}}^I (p^e)}{p^e} + \limsup_{e \to \infty} \frac{\nu_{{\mathfrak{b}_{\bullet}}}^I (p^e)}{p^e}. $$ Hence, the assertion follows.
\end{proof}

As an immediate consequence of \Cref{filwith}, \Cref{F-threshold-*-filtration} and \Cref{new_fil}, we get the following: 

\begin{corollary}
Let $\aa_\bullet$ and $\bb_\bullet$ be two $(I,p)$-admissible filtrations. Let $\cc_\bullet= \{ \aa_n \bb_n\}_{n\geq 0}$, $\dd_\bullet =  \{ \aa_n \cap \bb_n\}_{n\geq 0}$ and $\ee_\bullet= \{ \sum_{i=0}^n\aa_i\bb_{n-i}\}_{n\geq 0}$.  Then, 
\begin{enumerate}
\item  ${\mathcal{C}}^I(\cc_\bullet) \le \min \left\{ {\mathcal{C}}^I(\aa_\bullet) , {\mathcal{C}}^I(\bb_\bullet)  \right\}.$
\item Assume that $R$ is regular, then ${\mathcal{C}}^I(\cc_\bullet) \le \C^{I}(\dd_\bullet) \le \min \left\{ {\mathcal{C}}^I(\aa_\bullet) , {\mathcal{C}}^I(\bb_\bullet)  \right\}.$
\item   $\max \left\{ {\mathcal{C}}^I(\aa_\bullet) , {\mathcal{C}}^I(\bb_\bullet)  \right\} \le {\mathcal{C}}^I(\ee_\bullet) \le  {\mathcal{C}}^I(\aa_\bullet) +{\mathcal{C}}^I(\bb_\bullet) .$
\end{enumerate}
\end{corollary}

 \begin{proposition}\label{new-fil-fil}
Let  $\aa_\bullet$ be filtration and let $\bb_\bullet= \{ \aa_i^s\}_{i\geq 0}$. Then,  $s\;{\mathcal{C}}^I_{\pm}(\bb_\bullet) \le {\mathcal{C}}^I_{\pm}(\aa_\bullet) .$ Moreover, if $\aa_\bullet =\aa^\bullet$ for some non-zero proper ideal $\aa \subseteq \sqrt I,$ then $\C^I(\aa^\bullet)=s\; \C^I((\aa^s)^\bullet).$
\end{proposition}
\begin{proof} Observe that  $ s \nu_{\bb_\bullet}^I(p^e) \le \nu_{\aa_\bullet}^I(p^e)$  as if  $\bb_r \not\subseteq I^{[p^e]}$, then  $\aa_r^s \not\subseteq I^{[p^e]}$ which implies that $\aa_{rs} \not\subseteq I^{[p^e]} $. Now, using definition of $\mathcal{C}^I_{\pm}(\bullet)$, we get $s\;{\mathcal{C}}^I_{\pm}(\bb_\bullet) \le {\mathcal{C}}^I_{\pm}(\aa_\bullet) .$ 
 {The second assertion follows from \cite[Proposition 2.2 (iii)]{HMTW}.}
 \end{proof}

The above mentioned results are very useful in the computation of $F$-thresholds of filtrations of monomial ideals as is illustrated below.  A filtration $\aa_\bullet$ in a polynomial ring is said to be  a {\it monomial filtration}, if for each $i\ge 1$, $\aa_i$ is a monomial ideal. First, we prove an auxiliary lemma: 
\begin{lemma} \label{tech-con}
Let $R=\mathbb{K}[x_1,\ldots,x_n]$, where $\mathbb{K}$ is a field of prime characteristic $p>0$ and  $\aa_\bullet$ be a non-trivial monomial filtration in $R$ and let $I= (x_{1}^{m_{1}},\ldots, x_{n}^{m_{n}})$. Then, for all $e \ge 0$, $$\nu_{\aa_\bullet}^I(p^e) = \sup \left\{ r : x_{1}^{p^em_{1}-1}\cdots x_{n}^{p^em_{n}-1} \in \aa_r\right\}.$$
\end{lemma}
\begin{proof}
Let $r$ be any positive integer. We claim that $\aa_r \not\subseteq I^{[p^e]}$ if and only if $x_{1}^{p^em_{1}-1}\cdots x_{n}^{p^em_{n}-1} \in \aa_r.$  Suppose that $x_{1}^{p^em_{1}-1}\cdots x_{n}^{p^em_{n}-1} \in \aa_r.$  Then, $\aa_r \not\subseteq I^{[p^e]}$ as $x_{1}^{p^em_{1}-1}\cdots x_{n}^{p^em_{n}-1} \not\in I^{[p^e]}$. Conversely, suppose we have $\aa_r \not\subseteq I^{[p^e]}$ for some $r$. Therefore, there exists a monomial $u \in \aa_r \setminus I^{[p^e]}.$ Since $u \not \in  I^{[p^e]}$, for each $1 \leq j \leq n$, $u$ is not divisible by $x_{j}^{p^em_{j}}.$ This implies that if $u$ is divisible by $m$-th power of $x_{j}$, then $m \leq p^em_{j}-1.$ Therefore, $u$ divides $x_{1}^{p^em_{1}-1}\cdots x_{n}^{p^em_{n}-1}$ which implies that $x_{1}^{p^em_{1}-1}\cdots x_{n}^{p^em_{n}-1} \in \aa_r$. This completes the claim. Now,  the assertion follows from the definition of $\nu_{\aa_\bullet}^I(p^e)$. 
\end{proof}

{\begin{proposition}\label{int-fil-mon}
Let $R=\mathbb{K}[x_1,\ldots,x_n]$, where $\mathbb{K}$ is a field of prime characteristic $p>0$ and  $\aa_\bullet, \bb_\bullet$ be nontrivial monomial filtrations in $R$.  Let  $\dd_\bullet =  \{ \aa_n \cap \bb_n\}_{n\geq 0}$.  Then, $\C_{\pm}^{\mm}(\dd_\bullet) =\min\{ \C_{\pm}^{\mm}(\aa_\bullet), \C_{\pm}^{\mm}(\bb_\bullet)\},$ where $\mm=(x_1,\ldots,x_n).$
\end{proposition}
\begin{proof}
 First, we claim that $\nu_{\dd_\bullet}^{\mm}(p^e) = \min\{ \nu_{\aa_\bullet}^{\mm}(p^e), \nu_{\bb_\bullet}^{\mm}(p^e)\}$ for all $e$.  Let $r \le \min\{ \nu_{\aa_\bullet}^{\mm}(p^e), \nu_{\bb_\bullet}^{\mm}(p^e)\}$ be any. Then, by \Cref{tech-con},  $x_1^{p^e-1} \cdots x_n^{p^e-1} \in \aa_r\cap \bb_r$ which implies that $x_1^{p^e-1} \cdots x_n^{p^e-1} \in \dd_r$. Thus, applying \Cref{tech-con}, we get $r \le \nu_{\dd_\bullet}^{\mm}(p^e)$, and hence, $\nu_{\dd_\bullet}^{\mm}(p^e) \ge \min\{ \nu_{\aa_\bullet}^{\mm}(p^e), \nu_{\bb_\bullet}^{\mm}(p^e)\}$. Next, if $\min\{ \nu_{\aa_\bullet}^{\mm}(p^e), \nu_{\bb_\bullet}^{\mm}(p^e)\}= \infty$, then $\nu_{\dd_\bullet}^{\mm}(p^e) =\min\{ \nu_{\aa_\bullet}^{\mm}(p^e), \nu_{\bb_\bullet}^{\mm}(p^e)\}$. So, we assume that $\min\{ \nu_{\aa_\bullet}^{\mm}(p^e), \nu_{\bb_\bullet}^{\mm}(p^e)\}< \infty$. Suppose $r >  \min\{ \nu_{\aa_\bullet}^{\mm}(p^e), \nu_{\bb_\bullet}^{\mm}(p^e)\}$, then either $x_1^{p^e-1} \cdots x_n^{p^e-1}\not\in \aa_r$ or $x_1^{p^e-1} \cdots x_n^{p^e-1} \not\in  \bb_r$. Therefore, $x_1^{p^e-1} \cdots x_n^{p^e-1} \not\in \aa_r\cap \bb_r$. Thus, by \Cref{tech-con}, the claim follows.  Now, by \Cref{basic-prop} and applying limit to the sequence $\{\frac{\nu_{\dd_\bullet}^{\mm}(p^e)}{p^e}\}$, we get the desired result. 
\end{proof}}

In the following, we show how useful is  \Cref{int-fil-mon} in computation of symbolic $F$-thresholds of monomial ideals.
\begin{example}
Let $R=\mathbb{K}[x_1,\ldots,x_n]$, where $\mathbb{K}$ is a field of prime characteristic $p>0$ and $\mm=(x_1,\ldots,x_n).$ Let $\pp_1,\ldots,\pp_k$ be monomial primes in $R$ and $\omega_1,\ldots,\omega_k$ be positive integers. Let $\aa=\cap_{i=1}^k\pp^{\omega_i}$, then, $\aa$ is a monomial ideal. {By \Cref{exist}, we know that $\C^{\mm}(\aa^{(\bullet)})$ exists.   First note that $\aa^{(\bullet)} =\cap_{i=1}^k \left(\pp_i^{\omega_i}\right)^{(\bullet)}.$ }Therefore, by \Cref{int-fil-mon}, $\C^{\mm}(\aa^{(\bullet)}) =\min_{1 \le i \le k} \C^{\mm}(\left(\pp_i^{\omega_i}\right)^{(\bullet)}).$ So, it is enough to compute $\C^{\mm}(\left(\pp^{\omega}\right)^{(\bullet)}),$  where $\pp$ is  a monomial prime ideal in $R$ and $\omega$ is a positive integer. Note that $\left(\pp^{\omega}\right)^{(\bullet)}=\left(\pp^{\omega}\right)^{\bullet}$. Now, by \Cref{new-fil-fil}, $\C^{\mm}(\left(\pp^{\omega}\right)^{\bullet})=\frac{\C^\mm(\pp^\bullet)}{\omega}=\frac{\hgt(\pp)}{\omega}$, where the equality $\C^\mm(\pp^\bullet) = \hgt(\pp)$ is well known and easy to show. Thus, $\C^{\mm}(\aa^{(\bullet)}) =\min_{1 \le i \le k} \frac{\hgt(\pp_i)}{\omega_i}.$ 
\end{example}

\begin{example}
Let $G$ be a chordal graph on the vertex set $\{x_1,\ldots,x_n\}$, $J(G)$ be the cover ideal of $G$ in $R=\mathbb{K}[x_1,\ldots,x_n]$, where $\mathbb{K}$ is a field of prime characteristic $p>0$ and $\mm=(x_1,\ldots,x_n).$ We claim that $\C^{\mm}(J(G)^{\bullet})=\frac{\omega(G)}{\omega(G)-1}$, where $\omega(G)$ is the clique number of $G.$ Let $K_{m_1},\ldots,K_{m_k}$ be the maximal cliques of $G$. It is well known fact that $G$ can be obtained as clique sums of maximal cliques (cf. \cite[Proposition 5.5.1]{Diestel}).  Then, if follows from \cite[Theorem 4.9]{JKM22}, $J(G)^{\bullet} =\cap_{i=1}^k J(K_{m_i})^{\bullet}.$ Thus, by \Cref{int-fil-mon}, $\C^{\mm}(J(G)^{\bullet}) =\min_{1 \le i \le k} \C^{\mm}(J(K_{m_i})^{\bullet})=\min_{ 1 \le i \le k} \frac{m_i}{m_i-1}$ where the last equality follows from \Cref{sym-ord-exm}. Finally note that $\min_{ 1 \le i \le k} \frac{m_i}{m_i-1}=\frac{\omega(G)}{\omega(G)-1}$.
\end{example}

Next, we show that one of the inequality in \Cref{new_fil} becomes an equality when the filtrations are in polynomial rings in disjoint set of variables over the same field.

{\begin{theorem}
Let $R_1=\mathbb{K}[x_1,\ldots,x_n]$ and $R_2=\mathbb{K}[y_1,\ldots,y_m]$, where $\mathbb{K}$ is a field of prime characteristic $p>0$ and let $R=\mathbb{K}[x_1,\ldots,x_n,y_1,\ldots,y_m].$ Let $\aa_\bullet$ and $\bb_\bullet$ be two filtrations in $R_1$ and $R_2$ respectively. Let $I\subseteq R_1$ and $J \subseteq R_2$ be non-zero proper ideals. Let $\cc_\bullet= \{ \aa_n \bb_n\}_{n\geq 0}$ and $\ee_\bullet= \{ \sum_{i=0}^n\aa_i\bb_{n-i}\}_{n\geq 0}$. Then, 
\begin{enumerate}
\item $\mathcal{C}_\pm^{I+J}(\ee_\bullet) =\C_\pm^{I+J}(\aa_\bullet)+\C_\pm^{I+J}(\bb_\bullet)=\mathcal{C}_\pm^I(\aa_\bullet)+ \mathcal{C}_\pm^J(\bb_\bullet) ;$
\item $\mathcal{C}_\pm^{IJ}(\cc_\bullet) = \max\{\mathcal{C}_\pm^{I}(\aa_\bullet), \mathcal{C}_\pm^J(\bb_\bullet) \}.$
\end{enumerate}
\end{theorem}}
\begin{proof} $(1)$ {First, observe that $\aa_r \not\subseteq I^{[p^e]}$ if and only if $\aa_r \not\subseteq (I+J)^{[p^e]}=I^{[p^e]}+J^{[p^e]}. $ Therefore, $\nu^{I+J}_{\aa_\bullet}(p^e)=\nu^{I}_{\aa_\bullet}(p^e)$, and hence, by \Cref{basic-prop} $\C_\pm^{I+J}(\aa_\bullet)=\C_\pm^{I}(\aa_\bullet)$. Similarly, $\C_\pm^{I+J}(\bb_\bullet)=\C_\pm^{J}(\bb_\bullet)$. Now, it follows from \Cref{basic-prop} and from \Cref{new_fil} that  $\mathcal{C}_\pm^{I+J}(\ee_\bullet) \le \C_\pm^{I+J}(\aa_\bullet)+\C_\pm^{I+J}(\bb_\bullet)=\mathcal{C}_\pm^I(\aa_\bullet)+ \mathcal{C}_\pm^J(\bb_\bullet) .$ Fix a positive integer $e$, take $r \le\nu^I_{\aa_\bullet}(p^e)$ and $s \le\nu^J_{\bb_\bullet}(p^e) $ be any.  Then, $\ee_{r+s} \not\subseteq (I+J)^{[p^e]}$ as $\aa_r \bb_s \subseteq \ee_{r+s}$. Therefore, $\nu^I_{\aa_\bullet}(p^e)+\nu^J_{\bb_\bullet}(p^e) \le \nu^{I+J}_{\ee_\bullet}(p^e)$  for all $e$. Now, by \Cref{basic-prop}, $\mathcal{C}_\pm^I(\aa_\bullet)+ \mathcal{C}_\pm^J(\bb_\bullet) \le \mathcal{C}_\pm^{I+J}(\ee_\bullet).$ Hence, $\mathcal{C}_\pm^{I+J}(\ee_\bullet) =\C_\pm^{I+J}(\aa_\bullet)+\C_\pm^{I+J}(\bb_\bullet)=\mathcal{C}_\pm^I(\aa_\bullet)+ \mathcal{C}_\pm^J(\bb_\bullet) .$}
\par $(2)$
First, note that if $s > \max \{ \nu^I_{\aa_\bullet}(p^e), \nu^J_{\bb_\bullet}(p^e) \}$, then $\aa_s \subseteq I^{[p^e]}$ and $\bb_s \subseteq J^{[p^e]}$ which implies that $\cc_s \subseteq (IJ)^{[p^e]}.$ Also, if $r=\max \{ \nu^I_{\aa_\bullet}(p^e), \nu^J_{\bb_\bullet}(p^e) \}.$ Then, $\cc_r=\aa_r \bb_r \not\subseteq I^{[p^e]}J^{[p^e]}$. Consequently,  $\nu^{IJ}_{\cc_\bullet}(p^e)=\max\{\nu^I_{\aa_\bullet}(p^e),\nu^J_{\bb_\bullet}(p^e)\}$  for all $e$. Hence,  the assertion follows from \Cref{basic-prop}. 
\end{proof}

%%%%%%%%%%%%%%%%%%%%%%%%%%%%%%%%%%%%%%%%%%%%%%%%%%%%%%%%%%%%%

\section{$F$-thresholds of symbolic power filtrations}\label{sym-sec}
In this section, we focus on symbolic $F$-threshold of an ideal and its connection with the notion of symbolic $F$-split ideals, recently introduced in \cite{AJL21}. We prove that symbolic $F$-threshold of an ideal being big-height of that ideal is a sufficient condition for that ideal to be symbolic $F$-split. Throughout this section, we assume that $R$ is a regular ring of prime characteristic $p>0$.

We now obtain an upper bound for symbolic $F$-threshold of an ideal.
\begin{proposition}\label{sym-bound}
Let $I$ and $\aa$ be non-zero proper ideals of $R$ such that $\aa \subseteq \sqrt I.$ Then,
{$$\mathcal{C}^{\mathfrak{p}}(\aa^{(\bullet)}) \le \mathcal{C}^{\sqrt I}(\aa^{(\bullet)}) \le \mathcal{C}^{\sqrt \aa}(\sqrt{\aa}^{(\bullet)}) \le \text{big-height}(\sqrt\aa),$$ for any prime ideal $\mathfrak{p}$ containing $I$. In particular, $\mathcal{C}^{\sqrt I} (\aa^{\bullet}) \le \text{big-height}(\sqrt\aa).$}
\end{proposition}
\begin{proof}
It follows from Theorem \ref{exist} that symbolic $F$-threshold of $\aa$ exists with respect to any ideal whose radical contains $\aa$. For any prime ideal $\mathfrak{p}$ of $R$ that contains $I$, by Proposition \ref{basic-prop}, $\mathcal{C}^{\mathfrak{p}}(\aa^{(\bullet)}) \le \mathcal{C}^{\sqrt I}(\aa^{(\bullet)}) \le \mathcal{C}^{\sqrt \aa}({\aa}^{(\bullet)})$ as $\sqrt \aa \subseteq \sqrt I \subseteq \mathfrak{p}$.  Note that $\aa^{(\bullet)} \le \sqrt{\aa}^{(\bullet)}$. Consequently, by Theorem \ref{fin-fil}, $\mathcal{C}^{\sqrt \aa}({\aa}^{(\bullet)}) \le \mathcal{C}^{\sqrt \aa}(\sqrt{\aa}^{(\bullet)}).$ Now, it is enough to prove that if $\bb$ is a non-zero proper  radical ideal of $R$, then $\mathcal{C}^{\bb}(\bb^{(\bullet)}) \le \text{big-height}(\bb).$ Let $H=\text{big-height}(\bb).$ Then, by \cite[Lemma 2.4]{HH02}, $\bb^{(Hp^e)} \subseteq \bb^{[p^e]}$ for all $e \ge 0$. Therefore, $\nu_{\bb^{(\bullet)}}^{\bb}(p^e) \le Hp^e-1$ which implies that $\frac{\nu_{\bb^{(\bullet)}}^{\bb}(p^e)}{p^e} < H$ for all $e \ge 0$.  Consequently, $\mathcal{C}^{\bb}(\bb^{(\bullet)})=\sup_{e\ge 0} \frac{\nu_{\bb^{(\bullet)}}^{\bb}(p^e)}{p^e} \le H.$ Thus, the  assertions follow. 
\end{proof}
\begin{proposition}\label{sym-bound1}
Let $I$ and $\aa$ be non-zero proper ideals of $R$ such that $\sqrt{\aa} \subseteq  I.$ Then,
$$\mathcal{C}^{\mathfrak{p}}(\aa^{(\bullet)}) \le \mathcal{C}^{I}(\aa^{(\bullet)}) \le \mathcal{C}^{I}(\sqrt{\aa}^{(\bullet)}) \le \mathcal{C}^{\sqrt \aa}(\sqrt{\aa}^{(\bullet)}) \le \text{big-height}(\sqrt\aa),$$  for any prime ideal $\mathfrak{p}$ containing $I$. In particular, $\mathcal{C}^{ I} (\aa^{\bullet}) \le \text{big-height}(\sqrt\aa).$
\end{proposition}
\begin{proof}
The proof is similar to the proof of Proposition \ref{sym-bound}. 
\end{proof}
\begin{setup}\label{setup1}
Let $(R,\mm, \mathbb{K})$ be a {$F$-finite} regular ring which is either local or positively graded polynomial ring over $\mathbb{K}$ with homogeneous maximal ideal $\mm$. In case when $R$ is graded, elements and ideals considered are homogeneous.
\end{setup}

%\begin{definition}
%(cf. \cite[Definition 5.2]{AJL21}) Let $(R,\mm, \mathbb{K})$ be a regular ring with \Cref{setup1} and let $\aa$ be  a non-zero proper ideal of $R$. Then $\aa$ is called \emph{symbolic $F$-split} if  there exists a splitting $\phi: R^{1/p}\to R$ such that 
 %$\phi((\aa^{(np+1)})^{1/p})\subseteq \aa^{(n+1)}$ for all $n\in \mathbb{Z}_{\geq 0}$. 
%\end{definition}

If we replace $I$ by a maximal ideal containing $\aa$ in Proposition \ref{sym-bound}, then $\text{ht}(\aa)$ serve as an upper bound for symbolic $F$-threshold. We prove this in the following:
 
\begin{corollary} \label{cor:bound-of-F-threshold} Let $(R,\mm, \mathbb{K})$ be a regular ring with \Cref{setup1} and let $\aa$ be  a non-zero proper ideal of $R$. Then, $$\mathcal{C}^{\mathfrak{m}}(\mathfrak{a}^{\bullet})\leq \mathcal{C}^{\mathfrak{m}}(\mathfrak{a}^{({\bullet})})\leq  \text{ht} (\aa).$$
\end{corollary}

\begin{proof}
 Let $\mathfrak{p}$ be a minimal prime of $\aa$ such that $\text{ht}(\mathfrak{p}) =\text{ht}(\aa).$ Note that $\aa^{\bullet} \le \aa^{(\bullet)} \le \mathfrak{p}^{(\bullet)}.$ By Proposition \ref{basic-prop} and Theorem \ref{fin-fil}, $$ \mathcal{C}^{\mathfrak{m}}(\mathfrak{a}^{\bullet})\leq \mathcal{C}^{\mathfrak{m}}(\mathfrak{a}^{({\bullet})})\leq  \mathcal{C}^{\mathfrak{m}}(\mathfrak{p}^{(\bullet)}) \le \mathcal{C}^\mathfrak{p}(\mathfrak{p}^{(\bullet)}).$$ Now, applying Proposition \ref{sym-bound}, $\mathcal{C}^\mathfrak{p}(\mathfrak{p}^{(\bullet)}) \le \text{ht}(\mathfrak{p}).$ Hence the assertion follows. 
\end{proof}

Now, we give a series of results to show that the upper bound obtain in \Cref{cor:bound-of-F-threshold} is a sharp bound. First, we see some examples:
\begin{example}
In \cite[Section 6]{AJL21}, the authors studied various types of  determinantal ideals in prime characteristic $p>0$. They proved that  if $\aa$ is an ideal of minors of a generic matrix or an ideal of minors of a symmetric matrix or an ideal of Pfaffians of a skew symmetric matrix or an ideal of minors of a Hankel matrix, then $\aa$ is symbolic $F$-split for $p\gg 0$, for more details see \cite[Theorem 6.5, Theorem 6.13, Theorem 6.22, Theorem 6.32]{AJL21}. The main strategy used to prove their results is existence of a homogeneous polynomial $f\in \aa^{(\hgt(\aa))}$ and $f^{p-1} \not\in \mm^{[p]}$, see proofs in \cite[Section 6]{AJL21} for more details.  By the construction of the polynomial $f$ in \cite{AJL21}, it is easy to show that, for every $e>0$, $f^{p^e-1} \not\in \mm^{[p^e]}.$ Now, for each $e$, $f^{p^e-1} \in \left(\aa^{(\hgt(\aa))}\right)^{p^e-1} \subseteq \aa^{((p^e-1)\hgt(\aa))}$ and $f^{p^e-1} \not\in \mm^{[p^e]}.$ Consequently, $\aa^{((p^e-1)\hgt(\aa))} \not\subseteq \mm^{[p^e]}$ for all $e$, hence  $\hgt(a)(p^e-1) \le \nu_{\aa^{(\bullet)}}^{\mathfrak{m}}(p^e)$  for all $e> 0$.  Now, applying limit $e$ tends to infinity to $\frac{\nu_{{\aa^{(\bullet)}}}^{\mathfrak{m}}(p^e) }{p^e}$, we get  $C^{\mathfrak{m}}({\aa^{(\bullet)}})\ge \text{ht}(\aa).$  Thus 
from Corollary \ref{cor:bound-of-F-threshold}, $C^{\mathfrak{m}}({\aa^{(\bullet)}})=\text{ht}(\aa).$
\end{example}

{Next we show that $F$-thresholds and symbolic $F$-thresholds of $F$-K\"onig ideals are heights of that ideals. Before that we recall the definition of  $F$-K\"onig ideals.}
\begin{definition}\cite[Definition 5.14]{AJL21}
Let $(R,\mm, \mathbb{K})$ be a regular ring with \Cref{setup1} and let $\aa$ be  a non-zero proper ideal of $R$ with  $\text{ht}(\aa)=h$. Then, $\aa$ is said to be $F$-K\"onig if there exists a regular sequence $f_1 , \cdots, f_h$ in $\aa $ such that $R/(f_1 ,\ldots, f_h)$ is $F$-split.
\end{definition}

\begin{theorem}
\label{P:F-konig}
Let $(R,\mm, \mathbb{K})$ be a regular ring with \Cref{setup1} and let $\aa$ be a $F$-K\"onig ideal in $R$.  Then,  
 $$\mathcal{C}^\mathfrak{m}(\mathfrak{a}^{\bullet})=\mathcal{C}^\mathfrak{m}(\mathfrak{a}^{(\bullet)})=\text{ht}(\aa).$$
\end{theorem}

\begin{proof}
Since $\aa$ is an $F$-K\"onig ideal in $R$ with $\text{ht}(\aa) =h$, there exists a regular sequence $f_1 , \cdots, f_h$ in $\aa $ such that $R/(f_1 ,\ldots, f_h)$ is $F$-pure.  Then by \cite{Fedder}, $f_1^{p^e-1} \cdots f_h^{p^e-1} \not\in \mathfrak{m}^{[p^e]}.$ 
  
 Note that $f_1^{p^e-1} \cdots f_h^{p^e-1} \in \aa^{h(p^e-1)}$ for all $ e \ge 0.$ Thus, $\aa^{h(p^e-1)} \not\subseteq \mathfrak{m}^{[p^e]}$, and hence,  $h(p^e-1) \le \nu_{\aa^{\bullet}}^{\mathfrak{m}}(p^e)$  for all $e \ge 0$.  Applying limit $e$ tends to infinity to $\frac{\nu_{{\aa^{\bullet}}}^{\mathfrak{m}}(p^e) }{p^e}$, we get  $C^{\mathfrak{m}}({\aa^{\bullet}})\ge \text{ht}(\aa).$ Now the rest follows from Corollary \ref{cor:bound-of-F-threshold}.
\end{proof}

As an example, we obtain $F$-thresholds and symbolic $F$-thresholds of binomial edge ideals of traceable graphs.
\begin{example}\label{traceable}
 Let $G$ be a traceable {graph} on the vertex set ${\{1,\cdots,n\}}$. Consider  the binomial edge ideal $\mathcal{J}_G$ of $G$ in $R=\mathbb{K}[x_1,\ldots,x_n,y_1,\ldots,y_n]$, where $\mathbb{K}$ is an $F$-finite field of prime characteristic $p>0$. It follows from \cite{Adam23} that $\mathcal{J}_G$ is $F$-K\"onig, and by \cite[Proposition 5.2]{Oh11}$, \text{ht} (\mathcal{J}_G) =n-1$. Therefore, by  \Cref{P:F-konig}, $\C^{\mm}(\mathcal{J}_G^{\bullet}) =\C^{\mm}(\mathcal{J}_G^{(\bullet)}) =n-1.$  
\end{example}

\begin{example}
 Let $A$ and $B$ be two generic matrices of size $2\times 2$  or $3\times 3$ with entries in disjoint sets of variables. Let $J$ be the ideal generated by the entries of $AB-BA$ and $I$ the ideal generated by the off-diagonal entries of this matrix. Then the ideals $I$ and $J$ are  unmixed  and $F$-K\"{o}nig \cite{K17}, therefore symbolic $F$-split (\cite[Proposition 5.15]{AJL21}). Thus, by  \Cref{P:F-konig}, 
  $\C^{\mm}(I^{\bullet}) =\C^{\mm}(I^{(\bullet)}) = \hgt I$ and $\C^{\mm}(J^{\bullet}) =\C^{\mm}(J^{(\bullet)}) = \hgt J$
\end{example}

\begin{example}
{Let $k,l$ be positive integers. Let $\aa=(x^k-y, x^l-z) \subseteq \mathbb{K}[x,y,z]$ be the defining ideal of the curve $\mathbb{K}[t,t^k,t^l]$, where $\mathbb{K}$ is an $F$-finite field of positive characteristic $p$. Note that $x^k-y, x^l-z$ is a regular sequence in $\mathbb{K}[x,y,z]$ and $\mathbb{K}[x,y,z]/ \aa$ is $F$-pure. Thus, $\aa$ is $F $-K\"onig, and hence, by \Cref{P:F-konig}, 
 $$C^\mathfrak{m}(\mathfrak{a}^{\bullet})=C^\mathfrak{m}(\mathfrak{a}^{(\bullet)})=2.$$}
\end{example}

 So far, we have seen two families of ideals that have symbolic $F$-threshold same as height. In Section \ref{sec-hom-fil}, we study  symbolic $F$-threshold of monomial ideals with respect to monomial maximal ideal $\mm$ in more details. We will prove  that if $\aa$ is a  square-free monomial ideal, then its symbolic $F$-threshold is $\hgt (\aa)$.  However, this is not the case for every non-zero proper homogeneous ideal in $R$, i.e., the upper bound for symbolic $F$-threshold with respect to  $\mm$ is not always achieved. We illustrate this in the following examples. These examples might be known to experts but we include the proofs for the sake of the readers.
  \begin{example}\label{non-exam}
Let  $R=\mathbb{K}[x_1,\ldots,x_n]$, where $\mathbb{K}$ is a field of prime characteristic $p>0$.
\par $(1)$ Assume that $n$ is an even integer. Let $\aa=(x_1,\ldots,x_n)^{\frac{n}{2}}$. Then, it follows from \Cref{alpha} that $\C^{\mm}(\aa^{\bullet})=2$. Observe that $\aa^{(i)} =\aa^i$ for all $i \ge 0$. Therefore,  $\C^{\mm}(\aa^{(\bullet)})< \hgt(\aa)$ if $n \ge 3$.
\par $(2)$  Let $a_1, \ldots,a_n$ be positive integers, and let $\aa=(x_1^{a_1}, \cdots, x_n^{a_n})$. We claim that 
$\C^{\mathfrak{m}}(\aa^{(\bullet)})=\frac{1}{a_1}+\cdots+\frac{1}{a_n}$. Since $\mm$ is associated prime of $\aa$, by \cite[Lemma 2.2]{HJKN2023}, $\aa^{(i)}=\aa^i$ for all $i$.  By Lemma \ref{tech-con},  $\nu_{\aa^{(\bullet)}}^{\mathfrak{m}}(p^e) = \max \{ r : x_{1}^{p^e-1}\cdots x_{n}^{p^e-1} \in \aa^r\}.$  Let $r$ be a maximal positive integer such that $x_{1}^{p^e-1}\cdots x_{n}^{p^e-1} \in \aa^r$. Then, there exist non-negative integers $r_1,\ldots,r_n$ such that $\sum_{i=1}^n r_i =r$ and $x_1^{a_1r_1}\cdots x_n^{a_nr_n}$ divides $x_{1}^{p^e-1}\cdots x_{n}^{p^e-1}$. Therefore, $a_ir_i \le p^e-1$ for all $ 1 \le i \le n$ which implies that $r=\sum_{i=1}^n r_i \le (p^e-1)\sum_{i=1}^n \frac{1}{a_i}$. Thus, $\nu_{\aa^{(\bullet)}}^{\mathfrak{m}}(p^e) =\lfloor (p^e-1)\sum_{i=1}^n \frac{1}{a_i}\rfloor$ for all $e$. Consequently, $\C^{\mathfrak{m}}(\aa^{(\bullet)}) = \sum_{i=1}^n \frac{1}{a_i}$. Now, if  $a_i\ge 2$ for some $i$, then $\C^{\mathfrak{m}}(\aa^{(\bullet)}) = \sum_{i=1}^n \frac{1}{a_i}< \text{ht}(\aa)$.

\par $(3)$ Let $\aa$ be an ideal as in part $(2)$ with $a_i \ge n+1$ for all $i$, and  let $\bb=\aa+(x_1\cdots x_n)$. We claim that $\C^{\mm}(\bb^{(\bullet)})=1.$ Since $\mm$ is associated prime of $\bb$, by \cite[Lemma 2.2]{HJKN2023}, $\bb^{(i)}=\bb^i$ for all $i$. Since $\aa_i \ge n+1$ for all $i$, $(x_1\cdots x_n)^r$ is element of minimum degree in $\bb^r$.  Let $r$ be a maximal positive integer such that $x_{1}^{p^e-1}\cdots x_{n}^{p^e-1} \in \bb^r$. Then, there exist non-negative integers $r_1,\ldots,r_n, r_{n+1}$ such that $\sum_{i=1}^{n+1} r_i =r$ and $x_1^{a_1r_1+r_{n+1}}\cdots x_n^{a_nr_n+r_{n+1}}$ divides $x_{1}^{p^e-1}\cdots x_{n}^{p^e-1}$. Therefore, $a_ir_i \le p^e-1-r_{n+1}$ for all $ 1 \le i \le n$ which implies that $r=\sum_{i=1}^{n+1} r_i \le r_{n+1}+ (p^e-1-r_{n+1})\left(\sum_{i=1}^n \frac{1}{a_i} \right) \le p^e-1$. Thus, $\nu_{\aa^{(\bullet)}}^{\mathfrak{m}}(p^e) =p^e-1$ for all $e$. Now, the rest  follows from the definition.
\end{example}

We now characterize radical ideals whose symbolic $F$-thresholds are $\text{big-height}$ of that ideals in terms of ideal containment. 
\begin{theorem}\label{big-height}
Let $I$ and $\aa$ be non-zero proper radical ideals of $R$ with $\aa \subseteq I$ and $\text{big-height}(\aa)=H$.  Then, we have the followings:  \begin{enumerate} \item if $\aa^{(H(p^e-1))} \subseteq I^{[p^e]}$, then $\mathcal{C}^{I}(\aa^{(\bullet)})\le H-\frac{1}{p^e};$ 
\item $\mathcal{C}^{I}(\aa^{(\bullet)})=H$ if and only if $\aa^{(H(p^e-1))} \not\subseteq I^{[p^e]}$ for all  $e$;
\item if $(R,\mm, \mathbb{K})$ is an $F$-finite regular ring with \Cref{setup1} and $\mathcal{C}^{\mathfrak{m}}(\aa^{(\bullet)})=H$, then $\aa$ is unmixed and symbolic $F$-split. Moreover, in this case, $R/\aa$ is equi-dimensional and $F$-split. 
\end{enumerate}  
\end{theorem}
\begin{proof}
(1) Suppose that $\aa^{(H(p^{e_0}-1))} \subseteq I^{[p^{e_0}]}$ for some $e_0$. Take $k=H(p^{e_0}-1)-1$. By \cite[Lemma 2.6]{GH17}, $\aa^{(Hp^e+kp^e-H+1)} \subseteq \left( \aa^{(k+1)}\right)^{[p^e]}$ for all $e$. Thus, for all $e$, $\aa^{(Hp^e+kp^e-H+1)} \subseteq \left( \aa^{(k+1)}\right)^{[p^e]} =\left( \aa^{(H(p^{e_0}-1))}\right)^{[p^e]}\subseteq \left(I^{[p^{e_0}]}\right)^{[p^e]}=I^{[p^{e_0+e}]}. $ Consequently, we have $\nu_{\aa^{(\bullet)}}^{I}(p^{e_0+e}) \le Hp^e+kp^e-H$ which implies that $\frac{\nu_{\aa^{(\bullet)}}^{I}(p^{e_0+e})}{p^{e_0+e}} \le \frac{Hp^e+kp^e-H}{p^{e_0+e}}= \frac{Hp^{e_0+e}-p^e-H}{p^{e_0+e}}=H-\frac{1}{p^{e_0}}-\frac{H}{p^{e_0+e}} $ for all $e$.  Now, applying limit on both side, we get $\C^{I}(\aa^{(\bullet)}) \le H-\frac{1}{p^{e_0}}.$

\par (2) If $\C^{I}(\aa^{(\bullet)}) =H$, then it follows from $(1)$ that $\aa^{(H(p^e-1))} \not\subseteq I^{[p^e]}$ for all  $e$. Conversely, assume that $\aa^{(H(p^e-1))} \not\subseteq I^{[p^e]}$ for all  $e$. Then, $\nu_{\aa^{(\bullet)}}^{I}(p^{e}) \ge H(p^{e}-1)$ for all $e$, and hence, $\C^{I}(\aa^{(\bullet)}) \ge H.$ Now, by  \Cref{sym-bound}, $\C^{I}(\aa^{(\bullet)}) = H.$

\par (3) From $(2)$, we get $\aa^{(H(p-1))} \not\subseteq \mm^{[p]}.$ Thus, by \cite[Corollary 5.10]{AJL21}, $\aa$ is symbolic $F$-split, and hence, by \cite[Remark 4.4]{AJL21}, $R/\aa$ is $F$-split. Also, by \Cref{cor:bound-of-F-threshold},  $H=\mathcal{C}^{\mathfrak{m}}(\mathfrak{a}^{({\bullet})})\leq  \text{ht} (\aa)$ which implies that $\aa$ is an unmixed ideal.  Hence, the assertion follows. 
\end{proof}

\begin{corollary}\label{not-F-pure}
Let $(R,\mm, \mathbb{K})$ be a  $F$-finite regular ring with \Cref{setup1}. Let $\aa$ be a non-zero proper radical ideal of $R$ with $\text{big-height}(\aa)=H$ such that $R/\aa$ is not $F$-split. Then, $\mathcal{C}^{\mathfrak{m}}(\aa^{(\bullet)})\le H-\frac{1}{p}$.
\end{corollary}
\begin{proof}{Since $R/\aa$ is not $F$-split, by \cite[Remark 4.4]{AJL21}, $\aa$ is not symbolic $F$-split. Therefore, by \cite[Corollary 5.10]{AJL21}, $\aa^{(H(p-1))}  \subseteq \mm^{[p]}.$ Hence, by \Cref{big-height}, the assertion follows.}
\end{proof}
 
Big-height of ideal is needed in \Cref{not-F-pure} as there are ideals which are not symbolic $F$-split but have height as its symbolic $F$-threshold.  We illustrate in the following: 
\begin{example}
Let $G$ be either a cycle of length $5$ or a biclaw tree and $\mathbb{K}$ be a field of characteristic $2$. It follows from \cite{Adam23} that  $\mathcal{J}_G$ is $F$-K\"onig, and therefore, by \Cref{P:F-konig}, $\C^{\mm}(\mathcal{J}_G^{(\bullet)})=\hgt(\mathcal{J}_G)$. However, it follows from \cite{Mat} that $R/\mathcal{J}_G$ is not $F$-split.  Hence, by \cite[Remark 4.4]{AJL21}, $\mathcal{J}_G$ is not symbolic $F$-split. 
\end{example}

We have seen in \Cref{big-height} that if $\C^{\mm}(\aa^{(\bullet)}) =H$, then $R/\aa$ is equi-dimensional and $F$-split, and in \Cref{not-F-pure} that if $R/\aa$ is not $F$-split, then $\mathcal{C}^{\mathfrak{m}}(\aa^{(\bullet)})\le H-\frac{1}{p}$. So, it is natural to ask whether $\C^{\mm}(\aa^{(\bullet)}) =H$ if and only if $R/\aa$ is $F$-split. In the following, we illustrate that this is not even true when $R/\aa$ is strongly $F$-regular.

\begin{example} Let $\mathbb{K}$ be a perfect field of prime characteristic $p \ge 3$.  Let $R=\mathbb{K}[a,b,c,d]$, $\mm=(a,b,c,d)$.
Let $\aa$ be the ideal generated by $2\times 2$ minors of the following matrix:
\begin{center}
 $\begin{bmatrix}
  a^2&b&d\\
  c&a^2&b-d\\
 \end{bmatrix}
,$
\end{center} i.e., 
 $\aa= (a^4-bc,b^2-bd-a^2d, a^2b-a^2d-cd)$. It follows from 
 \cite[Proposition 4.3]{AS99} that $R/\aa$ is strongly $F$-regular, and hence, $F$-split. However, $\aa$ is  not symbolic $F$-split, by \cite[Example 5.13]{AJL21}. We claim that $\C^{\mm}(\aa^{\bullet}) = \frac{7}{4}$.  By \cite[Example 4]{GH17},  $\aa$ is a prime ideal of height two, and $\aa^s=\aa^{(s)}$ for all $s$.   
 Assume that $\deg(a)=1$ and $\deg(b)=\deg(c)=\deg(d) =2$. Let $f$ be a homogeneous element of $R$ such that $f \not\in \mm^{[p^e]}$. Therefore, there exists a monomial term, say $u=\beta a^{\alpha_1}b^{\alpha_2}c^{\alpha_3}d^{\alpha_4}$, in $f$ such that $u \not\in \mm^{[p^e]},$ where $\beta \in \mathbb{K}\setminus \{0\}.$ Consequently, $\alpha_i\le p^e-1$ for each $i$, and hence, $\deg(f)=\deg(u) =\alpha_1+2\alpha_2+2\alpha_3+2\alpha_3 \le 7(p^e-1). $ Let $s$ be positive integers such that $s \ge \frac{7p^e}{4}$.  Since $\aa$ is homogeneous ideal of $R$ generated by elements of  degree $4$, $\aa^s$ is generated by elements of degree $4s$. Thus, $\aa^s \subseteq \mm^{[p^e]}$ as $4s \ge 7p^e$ which implies that $\nu_{\aa^\bullet}^{\mm} (p^e) \le \frac{7p^e}{4}.$ Now, applying limit as $e$ tends to infinity to $\frac{\nu_{\aa^\bullet}^{\mm} (p^e)}{p^e}$, we get $\C^{\mm}(\aa^{\bullet})\le \frac{7}{4}.$ 
 
 Next, for $s=7\lfloor \frac{(p^e-1)}{4} \rfloor$, $\aa^s \not\subseteq \mm^{[p^e]}$ as $u=\beta a^{4 \lfloor \frac{(p^e-1)}{4} \rfloor} b^{4\lfloor \frac{(p^e-1)}{4} \rfloor}(cd)^{4\lfloor \frac{(p^e-1)}{4} \rfloor}$ is a monomial term in $$f=(a^4-bc)^{ \lfloor \frac{(p^e-1)}{4} \rfloor} (b^2-bd-a^2d)^{2\lfloor \frac{(p^e-1)}{4} \rfloor}(a^2b-a^2d-cd)^{4\lfloor \frac{(p^e-1)}{4} \rfloor}$$ for some $\beta \in \mathbb{K}\setminus \{0\}$ and $u \not\in \mm^{[p^e]}.$ Thus, $\nu_{\aa^\bullet}^{\mm}(p^e) \ge 7\lfloor \frac{(p^e-1)}{4} \rfloor$ for all $e$. Now, applying limit as $e$ tends to infinity to $\frac{\nu_{\aa^\bullet}^{\mm} (p^e)}{p^e}$, we get $\C^{\mm}(\aa^{\bullet})\ge \frac{7}{4}.$ Hence, $\C^{\mm}(\aa^{\bullet})= \frac{7}{4}.$
 \end{example}

%%%%%%%%%%%%%%%%%%%%%%%%%%

\section{$F$-thresholds of  filtrations  via valuations}\label{sec-hom-fil}
In this section, we study $F$-threshold of filtration in a polynomial ring with respect to the {unique} monomial maximal ideal {via valuations}.  Throughout this section, $R=\mathbb{K}[x_1,\ldots,x_n]$ is a polynomial ring (need not be a standard graded unless otherwise stated) over a field of prime characteristic $p>0$ and $\mathfrak{m}=(x_1,\ldots,x_n)$ is the monomial maximal ideal of $R$.  Since $R$ is regular, by \Cref{basic-prop}, in order to show the existence of $F$-threshold of a filtration $\aa_\bullet$ with respect to $\mathfrak{m}$, it is enough to show that the sequence $\left\{ \frac{\nu_{\aa_\bullet}^{\mathfrak{m}}(p^e)}{p^e}\right\} $ is a bounded sequence. We provide a general upper bound for $F$-threshold of a filtration with respect to $\mm$ in terms of skew-Waldschmidt constant of that filtration. {We also provide a formula for $F$-threshold of a monomial ideal in terms of its Rees valuations. Also, we obtain symbolic F-threshold  of any square-free monomial ideal in terms of the monomial valuations associated to its minimal primes}.

%Before presenting these results, we recall the definition of  (skew) Waldschmidt constant of a filtration from  \cite{HKN22}.  Let $v$ be a valuation of the quotient field $K$ of $R$. We say that a valuation of $K$ is {\it supported on $R$} if $v(x)\ge 0$ for all $x \in R\setminus \{0\}.$ We assume that all the valuation{s} considered in this section are supported on $R$.  For {an} ideal $I$ of $R$, set
%$$v(I) = \min \{v(x) ~\big|~ x \in I \setminus \{0\}\}.$$
%It can be seen that for any filtration  $\aa_\bullet$, $\{v(\aa_i)\}_{i \ge 1}$ is a non-negative sub-additive sequence. Thus, by Fekete's Lemma, the limit $\lim\limits_{n \rightarrow \infty} \frac{v(\aa_n)}{n}$ exists and is equal to $\inf\limits_{n \in \NN} \frac{v(\aa_n)}{n}$. The \emph{skew-Waldschmidt constant} of $\aa_\bullet$ with respect to $v$ is {defined} to be
%$$\vhat(\aa_\bullet) := \lim\limits_{n \rightarrow \infty} \frac{v(\aa_n)}{n} = \inf\limits_{n \in \NN} \frac{v(\aa_n)}{n}.$$ One can note that $0 \le \vhat(\aa_\bullet) \le \frac{v(\aa_n)}{n} \le v(\aa_1).$

%In general, the skew-Waldschmidt constant of a filtration $\aa_\bullet$ with respect to a valuation $v$ need not be positive, for example, see \cite[Example 2.10]{HKN22}. 
Here, we give an equivalent criterion for  {the} positivity of the skew-Waldschmidt constant of a filtration with respect to a monomial valuation in terms of $F$-threshold of that filtration. 

 %A valuation $v$ of $K$ is said to be a \emph{monomial valuation}, if for every non-zero polynomial 
%$f \in R$,  $v(f )=\min\{  v(u) : u \text{ is a monomial term in } f\} $. We refer the readers to \cite{HS06} for more details about valuations and for undefined terms on valuations. 

\begin{theorem}\label{skew-Wald}
Let $\aa_\bullet$ be a filtration of  non-zero proper ideals and $v$ be a monomial valuation of $K$. If $\vhat(\aa_\bullet)>0$, then  
$$\mathcal{C}^{\mathfrak{m}}(\aa_\bullet) \leq \frac{v(x_1 \cdots x_n)}{\vhat(\aa_\bullet)}.$$ Moreover, if $v(\mm)>0$ and {$\mathcal{C}^{\mathfrak{m}}(\aa_\bullet) $ exists}, then $\vhat(\aa_\bullet)>0$.  
\end{theorem}
\begin{proof}
Assume that  $\vhat(\aa_\bullet)>0$.  Suppose that $\aa_r \not\subseteq \mathfrak{m}^{[p^e]}.$ Then, there exists  a polynomial $f \in \aa_r$ such that $f \not \in \mathfrak{m}^{[p^e]}.$ Since $f \not\in \mm^{[p^e]}$, there is a monomial term of $f$, say $u$, which is not in $\mm^{[p^e]}$. This implies that exponent of each variable in $u$ is at most $p^e-1$, and therefore, $u$ divides $x_1^{p^e-1}\cdots x_n^{p^e-1}.$ Consequently, $v(x_1^{p^e-1}\cdots x_n^{p^e-1}) \ge v(u) \ge v(f) \ge v(\aa_r)$ which implies that $(p^e-1)v(x_1\cdots x_n) \ge v(\aa_r).$  Thus, if $\aa_r \not\subseteq \mathfrak{m}^{[p^e]},$ then $r \leq \frac{r(p^e-1)v(x_1\cdots x_n)}{v(\aa_r)} \le \frac{(p^e-1)v(x_1\cdots x_n)}{\vhat(\aa_\bullet)}$ as $\vhat(\aa_\bullet) = \inf_{r \ge 1}\frac{v(\aa_r)}{r}$. Therefore, for all $e$, $\frac{\nu_{\aa_\bullet}^{\mathfrak{m}}(p^e)}{p^e} \leq \left(1-\frac{1}{p^e}\right)\frac{v(x_1\cdots x_n)}{\vhat(\aa_\bullet)}$, and hence, by \Cref{basic-prop}, $\mathcal{C}^{\mathfrak{m}}(\aa_\bullet) \leq \frac{v(x_1 \cdots x_n)}{\vhat(\aa_\bullet)}<\infty.$

 Conversely, we assume that $v(\mm)>0$ and {$\C^{\mm}(\aa_\bullet)$ exists}. Therefore, for some positive integer $M$, $\aa_{Mp^e} \subseteq \mm^{[p^e]}$ for all $e$. Consequently, $v(\aa_{Mp^e}) \ge v(\mm^{[p^e]})=p^e\min\{v(x_i) : 1 \le i \le n\}=p^ev(\mm).$ Now, $\vhat(\aa_\bullet) =\lim_{e \to \infty}\frac{v(\aa_{Mp^e})}{Mp^e} \ge \frac{v(\mm)}{M}>0.$ Hence, the assertion follows. 
\end{proof}

Next, we provide a class of filtration that attain the bound  given in \Cref{skew-Wald}. 
\begin{proposition}\label{fam-val}
Let $v_1, \ldots, v_k$ be monomial valuations of $K$, and let $m_1, \ldots,m_k$ be positive real numbers. Let $$\aa_n =(f \in R~:~ v_i(f) \ge m_i n\text{ for all } 1 \le i \le k)$$ for all $n \ge 0$.   Then, $$\C^{\mm}(\aa_\bullet) =\min_{ 1 \le i \le k} \frac{v_i(x_1 \cdots x_n)}{\vhat_i(\aa_\bullet)}=\min_{1 \le i \le k} \frac{v_i(x_1 \cdots x_n)}{m_i}.$$ Furthermore, if $v_i$'s are rational monomial valuations and $m_i$'s are rational, then $\C^{\mm}(\aa_\bullet)$ is rational.
\end{proposition}
\begin{proof}
 First of all note that $v_i(\aa_n) \ge m_in$ for all $n$, and therefore, $\vhat_i(\aa_\bullet) \ge m_i.$ Thus, by \Cref{skew-Wald}, $\C^{\mm}(\aa_\bullet) \le \min_{ 1 \le i \le k} \frac{v_i(x_1 \cdots x_n)}{\vhat_i(\aa_\bullet)}=\min_{1 \le i \le k} \frac{v_i(x_1 \cdots x_n)}{m_i}.$ Take $\alpha=\min_{1 \le i \le k} \frac{v_i(x_1 \cdots x_n)}{m_i}.$ Then, for each $1 \le i \le k$ and $e\ge 1$, $v_i(x_1^{p^e-1}\cdots x_n^{p^e-1}) =(p^e-1)v_i(x_1 \cdots x_n) \ge m_i \lfloor (p^e-1)\alpha \rfloor$, which implies that $x_1^{p^e-1}\cdots x_n^{p^e-1} \in \aa_{\lfloor (p^e-1)\alpha \rfloor}$. Thus, by \Cref{tech-con}, $\nu_{\aa_\bullet}^{\mm}(p^e) \ge {\lfloor (p^e-1)\alpha \rfloor} $ for each $e$. Now, applying limit to the sequence $\frac{\nu_{\aa_\bullet}^{\mm}(p^e)}{p^e}$ we get that $\C^{\mm}(\aa_\bullet) \ge \alpha =\min_{1 \le i \le k} \frac{v_i(x_1 \cdots x_n)}{m_i}.$ Hence, the assertion follows.
\end{proof}

Next, we study $F$-thresholds of homogeneous filtrations when $R$ is standard graded in terms of the degree valuation. %The {\it degree valuation}, denoted by $\alpha$, is  defined as follows: $\alpha(f) =\deg(f)$ for all non-zero homogeneous polynomial $f$ in $R$. It is easy to note that $\alpha$ is a monomial valuation.
\begin{corollary}
{Assume that $R$ is standard graded.} Let $\aa_\bullet$ be a filtration of  non-zero proper homogeneous ideals. Then, $\ahat(\aa_\bullet)>0$ if and only if  $\C^{\mm}(\aa_\bullet) <\infty.$ Furthermore, in this case, 
$$\mathcal{C}^{\mathfrak{m}}(\aa_\bullet) \leq \frac{n}{\ahat(\aa_\bullet)}.$$ 
\end{corollary}
\begin{proof}
The assertion immediately follows from Theorem \ref{skew-Wald} and the fact that $\alpha$ is a monomial valuation {with} $\alpha(\mm)>0.$
\end{proof}

As an immediate consequence, we obtain an upper bound for $F$-threshold and symbolic $F$-threshold of a non-zero proper homogeneous ideal $\aa$. 
\begin{corollary}\label{ub-f}
Assume that $R$ is standard graded. Let $\aa$ be a non-zero proper homogeneous ideal in $R$. Then, 
\begin{enumerate}
\item $\mathcal{C}^{\mathfrak{m}}(\aa^{\bullet}) \leq \frac{n}{\alpha(\aa)}.$
\item $\mathcal{C}^{\mathfrak{m}}(\aa^{(\bullet)}) \leq \frac{n}{\ahat(\aa^{(\bullet)})}$.
\end{enumerate}
\end{corollary}
\begin{remark}
 The above upper bound for $F$-thresholds is known. When the ideal is principle, another proof is given in \cite[Proposition 2.2]{ZJJK22}.
\end{remark}

We now obtain an upper bound for $F$-threshold of a square-free monomial ideal in terms of  combinatorial invariants of associated hypergraph. 
\begin{corollary}\label{edge-f}
Assume that $R$ is standard graded. Let $\aa$ be a non-zero proper square-free monomial ideal in $R$ and let $\mathcal{H}$ be the associated hypergraph whose edge ideal is $\aa$, i.e., $I(\H)=\aa$. Then,
\begin{enumerate}
\item $\mathcal{C}^{\mathfrak{m}}(\aa^{\bullet}) \leq \frac{n}{d},$ where $d=\min\{ |e| ~ : ~ e \in E(\mathcal{H})\}.$ 
\item  $\mathcal{C}^{\mathfrak{m}}(\aa^{(\bullet)}) \leq \frac{n(\chi_f(\mathcal{H})-1)}{\chi_f(\mathcal{H})}$, where $\chi_f(\mathcal{H})$ is the fractional chromatic number of $\mathcal{H}$.
\end{enumerate}
\end{corollary}
\begin{proof}
The first assertion immediately follows from \Cref{ub-f} and the fact that $\alpha(\aa) =d$. The second part follows from \Cref{ub-f} and \cite[Theorem 4.6]{BCG16}.
\end{proof}

The bound obtained in \Cref{ub-f} and \Cref{edge-f} are sharp, as the following examples illustrate it:

\begin{example}
Assume that $R$ is standard graded.  Let $G$ be a Hamiltonian graph on the vertex set $\{1,\ldots,n\}$ and {$I_G$  be the corresponding edge ideal in $R$}. We claim that  $\mathcal{C}^{\mm}(I_G^{\bullet}) =\frac{n}{2}.$ In view of \Cref{edge-f}, it is enough to prove that  $\mathcal{C}^{\mm}(I_G^{\bullet}) \ge \frac{n}{2}.$ Since $G$  is a Hamiltonian graph, we have $(x_1 \cdots x_n)^2 \in I_G^n$. For any $e$, write $p^e-1 = 2q_e+r_e$ with $0 \le r_e \le 1$ and $q_e \ge 0$.   Consider $(x_1 \cdots x_n)^{p^e-1}= (x_1 \cdots x_n)^{2q_e+r_e} \in I_G^{nq_e+r_e}  $.  Thus, $I_G^{nq_e+r_e} \not\subseteq \mm^{[p^e]}$ which implies that $\nu_{I_G^{\bullet}}^{\mm}(p^e) \ge nq_e = \frac{n(p^e-1-r_e)}{2} \ge \frac{n(p^e-2)}{2}.$  {Now, by applying limit to the both sides of the inequality, we get $\displaystyle \mathcal{C}^{\mm}(I_G^\bullet) =\lim_{e \to \infty} \frac{\nu_{I_G^{\bullet}}^{\mm}(p^e)}{p^e} \ge \lim_{e \to \infty} \frac{n(p^e-2)}{2p^e} =\frac{n}{2}. $ Hence, the assertion follows.}\end{example}

\begin{example} Assume that $R$ is standard graded. It is well known that the $F$-threshold  of edge ideal of a hypergraph is equal to the fractional matching number of that graph, see \cite{How01, Adam23}. In particular, if $\mathcal{H}$ is a hypergraph on the vertex set $\{1,\ldots, n\},$ then   $\mathcal{C}^{\mm}(I_{\mathcal{H}}^{\bullet}) \ge m(\mathcal{H}),$ where $m(\mathcal{H})$ is the matching number of $\mathcal{H}$. Thus, if $G$ is a graph with perfect matching, then by  \Cref{edge-f}, $\mathcal{C}^{\mm}(I_G^{\bullet}) = m(G)=\frac{n}{2},$ as every vertex is part of some edge in the perfect matching.   
\end{example}

\begin{example}
Assume that $R$ is standard graded.  Let $G=C_{n}$ be an odd cycle on the vertex set ${\{1,\ldots,n\}}$ and $J(G)$ be the vertex cover ideal of $G$. It follows from the proof of \cite[Theorem 4.7]{JKM22} that $\ahat(J(G)^{(\bullet)}) =\frac{n}{2}$ and $(x_1\cdots x_n)^s \in J(G)^{(2s)} $ for all $s$. Consequently, by \Cref{ub-f},  $\mathcal{C}^{\mm}(J(G)^{(\bullet)}) \le 2.$  One can see that  $(x_1\cdots x_n)^{p^e-1} \in J(G)^{(2(p^e-1))} \setminus \mm^{[p^e]}$. Therefore, $ \nu_{J(G)^{(\bullet)}}^{\mm}(p^e) \ge 2(p^e-1). $ Now, by \Cref{basic-prop}, $\mathcal{C}^{\mm}(J(G)^{(\bullet)}) =\sup_{e \ge 0} \frac{\nu^{\mm}_{J(G)^{(\bullet)}}(p^e)}{p^e} \ge \sup_{e \ge 0} \frac{2(p^e-1)}{p^e} \ge 2. $ Thus, $\mathcal{C}^{\mathfrak{m}}(J(G)^{(\bullet)}) = \frac{n}{\ahat(J(G)^{(\bullet)})}$.
\end{example}

\begin{lemma}\label{ub-val}
Let $\aa_\bullet$ be a filtration of monomial ideals in $R$. Let $v$ be a valuation of quotient field of $R$. If $\vhat(\aa_\bullet)>0$, then $$\mathcal{C}^{\mathfrak{m}}(\aa_\bullet) \leq \frac{v(x_1\cdots x_n)}{\vhat(\aa_\bullet)}.$$  Moreover, if $v(\mm)>0$  and {$\mathcal{C}^{\mathfrak{m}}(\aa_\bullet)$ exists}, then $\vhat(\aa_\bullet)>0$.
\end{lemma}
\begin{proof}
The proof is similar to the proof of \Cref{skew-Wald}. 
\end{proof}
 
 Next, we study $F$-thresholds and symbolic $F$-thresholds of monomial ideals in terms of Rees valuations. %Before that we recall  some facts about Rees valuations of an ideal.

 %Every ideal $\aa$ in a Noetherian domain $R$ has a unique finite set (up to equivalence) of Rees valuations, see \cite[Theorem 10.1.6, Theorem 10.2.2]{HS06}. If $\aa$ is a monomial ideal in a polynomial ring $R$, then every Rees valuation of $\aa$ is a monomial valuation,  \cite[Proposition 10.3.4]{HS06}. 
 %Let $\{v_1, \cdots, v_k\}$ be Rees valuations of $\aa$. Then,  
 %$\overline{\aa^n} = \{r\in R : v_i(r) \geq nv_i(\aa), \text{ for  all } 1\leq i\leq k\}$ for every $n$. 

%As, we have seen already in \Cref{fam-val} that the bound obtain in Theorem \ref{skew-Wald} or in \Cref{ub-val} is a sharp bound. Now, we see a few more instance where it is again sharp. Recall that for a monomial ideal every Rees valuation is a monomial valuation. We obtain $F$-threshold of ordinary power filtration of a monomial ideal in terms of Rees valuations.

\begin{theorem}\label{Rees-valuation}
 Let $\aa$ be a non-zero proper monomial ideal in $R$. Then,   $$C^{\mathfrak{m}}(\aa^{\bullet})= \min \left\{\frac{v(x_1 \cdots x_n)}{v(\aa)}: v \in \text{RV}(\aa) \right\},$$ where $\text{RV}(\aa) $ denote the set of Rees valuations of $\aa.$ Moreover, $C^{\mathfrak{m}}(\aa^{\bullet})$ is a rational number.
\end{theorem}

\begin{proof}
Consider the filtration $\overline{\aa^{\bullet}} =\{ \overline{\aa^i}\}_{i \ge 0}$. It follows from   \Cref{int-f} that  $C^{\mathfrak{m}}(\aa^{\bullet})=C^{\mathfrak{m}}(\overline{\aa^{\bullet}}).$  Thus it is enough to prove that $$C^{\mathfrak{m}}(\overline{\aa^{\bullet}})= \min \left\{\frac{v(x_1 \cdots x_n)}{v(\aa)}: v \in \text{RV}(\aa) \right\} .$$  Let $w \in \text{RV}(\aa)$ such that $$\frac{w(x_1 \cdots x_n)}{w(\aa)} =\min \left\{\frac{v(x_1 \cdots x_n)}{v(\aa)}: v \in \text{RV}(\aa) \right\} .$$ Let $r$ be a non-negative  integer.  Then,  $x_{1}^{p^e-1}\cdots x_{n}^{p^e-1} \in \overline{\aa^r}$ if and only if  $v(x_{1}^{p^e-1}\cdots x_{n}^{p^e-1} ) \ge r v( \aa)$ for all $v \in \text{RV}(\aa)$  which happens if and only if  $r \le \frac{w(x_1\cdots x_n)}{{w}(\aa)}(p^e-1)$. Thus, by \Cref{tech-con}, $\nu_{\overline{\aa^\bullet}}^{\mathfrak{m}}(p^e) = \left \lfloor \frac{w(x_1\cdots x_n)}{{w}(\aa)}(p^e-1) \right \rfloor.$ Now, by \Cref{basic-prop} and 
applying limit $e$ tends to infinity to $\frac{\nu_{\overline{\aa^\bullet}}^{\mathfrak{m}}(p^e) }{p^e}$, we get  $$C^{\mathfrak{m}}(\overline{\aa^{\bullet}})=  \frac{w(x_1\cdots x_n)}{{w}(\aa)}.$$ Hence, the assertion follows. 
\end{proof}

 Using Theorem \ref{Rees-valuation}, one can characterize monomial ideals for which the inequality in  \Cref{ub-f} $(1)$ becomes an equality.

\begin{remark}
Assume that $R$ is standard graded. Let $\aa$ be a monomial ideal in $R$. We claim that $\C^{\mm}(\aa^\bullet)= \frac{n}{\alpha(\aa)}$  if and only if $(x_1 \cdots   x_n)^{\alpha(\aa)} \in \overline{\aa^n}.$ First, we assume that $\C^{\mm}(\aa^\bullet)= \frac{n}{\alpha(\aa)}$. By \Cref{Rees-valuation},   $\frac{n}{\alpha(\aa)}=\C^{\mm}(\aa^\bullet)\le \frac{v(x_1 \cdots x_n)}{v(\aa)}$ for all  $v \in \text{RV}(\aa) $. Therefore, $v((x_1\cdots x_n)^{\alpha(\aa)}) \ge v(\aa^n)$ for all $v \in \text{RV}(\aa) $. Consequently, by valuation criteria of integral closure (Remark~\ref{valuative-criterion}), $(x_1 \cdots x_n)^{\alpha(\aa)} \in \overline{\aa^n}.$ 
Conversely, we have $(x_1 \cdots x_n)^{\alpha(\aa)} \in \overline{\aa^n}.$ Then, $v((x_1\cdots x_n)^{\alpha(\aa)}) \ge v(\aa^n)$ for all $v \in \text{RV}(\aa) $. Equivalently, $\frac{n}{\alpha(\aa)}\le \frac{v(x_1 \cdots x_n)}{v(\aa)}$ for all  $v \in \text{RV}(\aa) $. Now, by Theorem \ref{Rees-valuation} and \Cref{ub-f}, $\C^{\mm}(\aa^\bullet) = \frac{n}{\alpha(\aa)}$. Hence, the claim follows. 
\end{remark}

In \Cref{cor:bound-of-F-threshold}, we have seen that $\hgt(\aa)$ is an upper bound of $\C^{\mm}(\aa^\bullet)$. In the following remark, we provide a necessary and a sufficient condition which impose  monomial ideals to have maximal $F$-threshold.

\begin{remark}
Let $\aa$ be a monomial ideal in $R$. We claim that $\C^{\mm}(\aa^\bullet)= \hgt(\aa)$  if and only if $x_1 \cdots   x_n \in \overline{\aa^{\hgt(\aa)}}.$ {Note that}  $\C^{\mm}(\aa^\bullet)= \hgt(\aa)$ if and only if $\hgt(\aa)\le \frac{v(x_1 \cdots x_n)}{v(\aa)}$ for all  $v \in \text{RV}(\aa) $ (by \Cref{Rees-valuation}), {if and only if} $v(x_1\cdots x_n) \ge v(\aa^{\hgt(\aa)})$ for all $v \in \text{RV}(\aa) $ which is equivalent to  $x_1 \cdots x_n \in \overline{\aa^{\hgt(\aa)}}$, by valuation criteria of integral closure (Remark ~\ref{valuative-criterion}). Hence, the claim follows. 
\end{remark}

We now obtain $F$-threshold of Noetherian monomial filtration  in terms of Rees valuations.

\begin{theorem}\label{val-noe-fil}
Let $\aa_\bullet$ be a  filtration of monomial ideals in $R$ such that $\R(\aa_\bullet)$ is Noetherian. Then $$\C^{\mm}(\aa_\bullet) =\min_{v \in \text{RV}(\aa_r)} \frac{v(x_1\cdots x_n)}{\vhat(\aa_\bullet)},$$ where $r$ is a positive integer such that  the $r$-th Veronese subalgebra $\R^{[r]}(\aa_\bullet):=\bigoplus_{i\ge 0} \aa_{ri}t^{ri}$ of $\R(\aa_\bullet)$ is a standard graded $R$-algebra.   Moreover, $C^{\mathfrak{m}}(\aa_{\bullet})$ is a rational number.
\end{theorem}
\begin{proof}
It follows from \Cref{noe-fil} that $\C^{\mm}(\aa_\bullet)=r\C^{\mm}(\aa_r^\bullet)$, and from the proof, it follows that  $\R^{[r]}(\aa_\bullet)$ is a standard graded $R$-algebra. Therefore, by \cite[Lemma 2.11]{HKN22},  for any valuation $v$ of $K$, $\vhat(\aa_\bullet)=\frac{v(\aa_r)}{r}.$ The rest now follows from \Cref{Rees-valuation}.
\end{proof}

As a consequence, we obtain symbolic $F$-threshold of  monomial ideals in terms of Rees valuations.

\begin{corollary}\label{sym-val}
 Let $\aa$ be a non-zero proper monomial ideal in $R$. Then, 
$$C^{\mathfrak{m}}(\aa^{(\bullet)})= \min_{v \in \text{RV}(\aa_r)} \left\{\frac{v(x_1 \cdots x_n)}{\vhat(\aa^{(\bullet)})}: v \in \text{RV}(\aa^{(r)}) \right\},$$ where $r$ is a positive integer such that $\R^{[r]}(\aa^{(\bullet)})$ is a standard graded $R$-algebra. 
 In particular,  $C^{\mathfrak{m}}(\aa^{(\bullet)})$ is a rational number.
\end{corollary}
\begin{proof}
It follows from  \cite[Theorem 3.2]{HHT07} that $\R(\aa^{(\bullet)})$ is Noetherian. Now, the assertion follows from \Cref{val-noe-fil}.
\end{proof}

Next, we obtain symbolic $F$-threshold of a square-free monomial ideal in term of monomial valuations given by minimal primes of that ideal. 
\begin{theorem}\label{sym-mon}
 Let $\aa$ be a square-free monomial ideal in $R$. Then, 
 $$C^{\mathfrak{m}}(\aa^{(\bullet)})= \min \left\{\frac{v_\mathfrak{p}(x_1 \cdots x_n)}{\vhat_\mathfrak{p}(\aa^{(\bullet)})}: \mathfrak{p} \in \text{Min}(\aa) \right\}=\text{ht}(\aa) ,$$ where for $\mathfrak{p} \in \text{Min}(\aa)$, $v_\mathfrak{p} $ is a monomial valuation defined as $v_\mathfrak{p}(x_1^{m_1} \cdots x_n^{m_n}) =\sum_{x_i \in \mathfrak{p}} m_i$ for all monomials $x_1^{m_1} \cdots x_n^{m_n}$ in $R$.
\end{theorem}
\begin{proof}
Note that  $v_\mathfrak{p}(x_1\cdots x_n) =\sum_{x_i \in \mathfrak{p}} 1 = \text{ht}(\mathfrak{p})$ for any $\mathfrak{p} \in \text{Min}(\aa)$.  Since $\aa^{(s)} \subseteq \mathfrak{p}^s$, $s= v_\mathfrak{p}(\mathfrak{p}^s) \le v_\mathfrak{p}(\aa^{(s)})$ for all $\mathfrak{p} \in \text{Min}(\aa).$ Therefore, $\vhat_\mathfrak{p}(\aa^{(\bullet)}) =\inf_{s\ge 1} \frac{v_\mathfrak{p}(\aa^{(s)})}{s} \ge 1$ for every $\mathfrak{p} \in \text{Min}(\aa).$ By \Cref{ub-val},  $$C^{\mathfrak{m}}(\aa^{(\bullet)})\le  \min \left\{\frac{v_\mathfrak{p}(x_1 \cdots x_n)}{\vhat_\mathfrak{p}(\aa^{(\bullet)})}: \mathfrak{p} \in \text{Min}(\aa) \right\}\le \text{ht}(\aa). $$  Next, take $r = \text{ht}(\aa)(p^e -1).$   Since  $x_{1}^{p^e-1}\cdots x_{n}^{p^e-1} \in  \mathfrak{p}^{\text{ht}(\mathfrak{p})(p^e-1)}$ for all $\mathfrak{p} \in \text{Min}(\aa)$,   $x_{1}^{p^e-1}\cdots x_{n}^{p^e-1} \in {\aa^{(r)}}$. Thus,  by \Cref{tech-con}, $\nu_{{\aa^{(\bullet)}}}^{\mathfrak{m}}(p^e) \ge \text{ht}(\aa) (p^e-1)$ for all $e \ge 0$.  Now, by \Cref{basic-prop} and 
applying limit $e$ tends to infinity to $\frac{\nu_{{\aa^{(\bullet)}}}^{\mathfrak{m}}(p^e) }{p^e}$, we get that $C^{\mathfrak{m}}({\aa^{(\bullet)}})\ge \text{ht}(\aa).$ Hence, the assertion follows. 
\end{proof}

One could observe that if $\aa$ is not a square-free monomial ideal, then Theorem \ref{sym-mon} may not hold, see \Cref{non-exam}. We end this section by characterizing  square-free monomial ideals for which the inequality in  \Cref{edge-f} $(2)$ becomes an equality. 

\begin{example}
Assume that $R$ is standard graded.  Let $\aa$ be a square-free monomial ideal in $R$ and  $\H$ be the associated hypergraph on the vertex set ${\{1,\ldots, n\}}$, i.e., $I_\H=\aa$.  We claim that $\C^{\mm}(\aa^{(\bullet)}) = \frac{n(\chi_f(\mathcal{H})-1)}{\chi_f(\mathcal{H})}$ if and only if  $\chi_f(\H)=\frac{n}{n-\tau(\H)}$, where $\tau(\H)$ is the minimum among cardinality of vertex covers of $\H$. Note that hypergraph with  $\chi_f(\H)=\frac{n}{n-\tau(\H)}$ exists, for example, any vertex transitive graph. The claim immediately follows from \Cref{sym-mon} and the fact that the minimal primes of ${I_\H}$ are in one to one correspondence with the minimal vertex covers of ${\H}$. 
\end{example}

%%%%%%%%%%%%%%%%%%%%%%%%%%%%%%%%%

%%%%%%%%%%%%%%%%%%%%%%%%%%%%%%%%%%

\bibliographystyle{plain}  %% or 
	\bibliography{bilbo}

\begin{thebibliography}{10}

\bibitem{BS}
Bhargav Bhatt and Anurag~K. Singh.
\newblock The {$F$}-pure threshold of a {C}alabi-{Y}au hypersurface.
\newblock {\em Math. Ann.}, 362(1-2):551--567, 2015.

\bibitem{MMK09}
Manuel Blickle, Mircea Musta\c{t}\u{a}, and Karen~E. Smith.
\newblock {$F$}-thresholds of hypersurfaces.
\newblock {\em Trans. Amer. Math. Soc.}, 361(12):6549--6565, 2009.

\bibitem{MMK08}
Manuel Blickle, Mircea Musta\c{t}\v{a}, and Karen~E. Smith.
\newblock Discreteness and rationality of {$F$}-thresholds.
\newblock volume~57, pages 43--61. 2008.
\newblock Special volume in honor of Melvin Hochster.

\bibitem{BCG16}
Cristiano Bocci, Susan Cooper, Elena Guardo, Brian Harbourne, Mike Janssen, Uwe
  Nagel, Alexandra Seceleanu, Adam Van~Tuyl, and Thanh Vu.
\newblock The {W}aldschmidt constant for squarefree monomial ideals.
\newblock {\em J. Algebraic Combin.}, 44(4):875--904, 2016.

\bibitem{TK15}
Takahiro Chiba and Kazunori Matsuda.
\newblock Diagonal {F}-thresholds and {F}-pure thresholds of {H}ibi rings.
\newblock {\em Comm. Algebra}, 43(7):2830--2851, 2015.

\bibitem{C91}
Steven~Dale Cutkosky.
\newblock Symbolic algebras of monomial primes.
\newblock {\em J. Reine Angew. Math.}, 416:71--89, 1991.

\bibitem{AJL21}
Alessandro {De Stefani}, Jonathan {Monta{\~n}o}, and Luis
  {N{\'u}{\~n}ez-Betancourt}.
\newblock {Blowup algebras of determinantal ideals in prime characteristic}.
\newblock {\em arXiv e-prints}, page arXiv:2109.00592, 2023.

\bibitem{DNP18}
Alessandro De~Stefani, Luis N\'{u}\~{n}ez Betancourt, and Felipe P\'{e}rez.
\newblock On the existence of {$F$}-thresholds and related limits.
\newblock {\em Trans. Amer. Math. Soc.}, 370(9):6629--6650, 2018.

\bibitem{Diestel}
Reinhard Diestel.
\newblock {\em Graph theory}, volume 173 of {\em Graduate Texts in
  Mathematics}.
\newblock Springer-Verlag, Berlin, third edition, 2005.

\bibitem{Fedder}
Richard Fedder.
\newblock {$F$}-purity and rational singularity.
\newblock {\em Trans. Amer. Math. Soc.}, 278(2):461--480, 1983.

\bibitem{GH17}
Elo\'{\i}sa Grifo and Craig Huneke.
\newblock Symbolic powers of ideals defining {F}-pure and strongly {F}-regular
  rings.
\newblock {\em Int. Math. Res. Not. IMRN}, (10):2999--3014, 2019.

\bibitem{HJKN2023}
Huy~T\`ai H\`a, A.~V. Jayanthan, Arvind Kumar, and Hop~D. Nguyen.
\newblock Binomial expansion for saturated and symbolic powers of sums of
  ideals.
\newblock {\em J. Algebra}, 620:690--710, 2023.

\bibitem{HKN22}
Huy~T\`ai H\`a, Arvind Kumar, Hop~D. Nguyen, and Thai~Thanh Nguyen.
\newblock Resurgence number of graded families of ideals, 2023.

\bibitem{Ha-Nguyen}
Huy~Tai Ha and Thai~Thanh Nguyen.
\newblock Newton-okounkov body, rees algebra, and analytic spread of graded
  families of monomial ideals, 2023.

\bibitem{HY}
Nobuo Hara and Ken-Ichi Yoshida.
\newblock A generalization of tight closure and multiplier ideals.
\newblock {\em Trans. Amer. Math. Soc.}, 355(8):3143--3174, 2003.

\bibitem{Her}
Daniel~J. Hern\'{a}ndez.
\newblock {$F$}-pure thresholds of binomial hypersurfaces.
\newblock {\em Proc. Amer. Math. Soc.}, 142(7):2227--2242, 2014.

\bibitem{HHT07}
J\"{u}rgen Herzog, Takayuki Hibi, Freyja Hreinsd\'{o}ttir, Thomas Kahle, and
  Johannes Rauh.
\newblock Binomial edge ideals and conditional independence statements.
\newblock {\em Adv. in Appl. Math.}, 45(3):317--333, 2010.

\bibitem{D09}
Daisuke Hirose.
\newblock Formulas of {F}-thresholds and {F}-jumping coefficients on toric
  rings.
\newblock {\em Kodai Math. J.}, 32(2):238--255, 2009.

\bibitem{DKK14}
Daisuke Hirose, Kei-ichi Watanabe, and Ken-ichi Yoshida.
\newblock {$F$}-thresholds versus {$a$}-invariants for standard graded toric
  rings.
\newblock {\em Comm. Algebra}, 42(6):2704--2720, 2014.

\bibitem{HH02}
Melvin Hochster and Craig Huneke.
\newblock Comparison of symbolic and ordinary powers of ideals.
\newblock {\em Invent. Math.}, 147(2):349--369, 2002.

\bibitem{How01}
J.~A. Howald.
\newblock Multiplier ideals of monomial ideals.
\newblock {\em Trans. Amer. Math. Soc.}, 353(7):2665--2671, 2001.

\bibitem{H82}
Craig Huneke.
\newblock On the finite generation of symbolic blow-ups.
\newblock {\em Math. Z.}, 179(4):465--472, 1982.

\bibitem{HMTW}
Craig Huneke, Mircea Musta\c{t}\u{a}, Shunsuke Takagi, and Kei-ichi Watanabe.
\newblock F-thresholds, tight closure, integral closure, and multiplicity
  bounds.
\newblock volume~57, pages 463--483. 2008.
\newblock Special volume in honor of Melvin Hochster.

\bibitem{HS06}
Craig Huneke and Irena Swanson.
\newblock {\em Integral closure of ideals, rings, and modules}, volume 336 of
  {\em London Mathematical Society Lecture Note Series}.
\newblock Cambridge University Press, Cambridge, 2006.

\bibitem{JKM22}
A.~V. Jayanthan, Arvind Kumar, and Vivek Mukundan.
\newblock On the resurgence and asymptotic resurgence of homogeneous ideals.
\newblock {\em Math. Z.}, 302(4):2407--2434, 2022.

\bibitem{K17}
Zhibek Kadyrsizova.
\newblock Nearly commuting matrices.
\newblock {\em J. Algebra}, 497:199--218, 2018.

\bibitem{ZJJK22}
Zhibek Kadyrsizova, Jennifer Kenkel, Janet Page, Jyoti Singh, Karen~E. Smith,
  Adela Vraciu, and Emily~E. Witt.
\newblock Lower bounds on the {$F$}-pure threshold and extremal singularities.
\newblock {\em Trans. Amer. Math. Soc. Ser. B}, 9:977--1005, 2022.

\bibitem{Adam23}
Adam {LaClair}.
\newblock {Invariants of Binomial Edge Ideals via Linear Programs}.
\newblock {\em arXiv e-prints}, page arXiv:2304.13299, April 2023.

\bibitem{Mat}
Kazunori Matsuda.
\newblock Weakly closed graphs and {$F$}-purity of binomial edge ideals.
\newblock {\em Algebra Colloq.}, 25(4):567--578, 2018.

\bibitem{KMK10}
Kazunori Matsuda, Masahiro Ohtani, and Ken-ichi Yoshida.
\newblock Diagonal {$F$}-thresholds on binomial hypersurfaces.
\newblock {\em Comm. Algebra}, 38(8):2992--3013, 2010.

\bibitem{F-threshold-det-ring}
Lance~Edward Miller, Anurag~K. Singh, and Matteo Varbaro.
\newblock The {$F$}-pure threshold of a determinantal ideal.
\newblock {\em Bull. Braz. Math. Soc. (N.S.)}, 45(4):767--775, 2014.

\bibitem{MTW}
Mircea Musta\c{t}\v{a}, Shunsuke Takagi, and Kei-ichi Watanabe.
\newblock F-thresholds and {B}ernstein-{S}ato polynomials.
\newblock In {\em European {C}ongress of {M}athematics}, pages 341--364. Eur.
  Math. Soc., Z\"{u}rich, 2005.

\bibitem{BI20}
Luis N\'{u}\~{n}ez Betancourt and Ilya Smirnov.
\newblock Hilbert-{K}unz multiplicities and {$F$}-thresholds.
\newblock {\em Bol. Soc. Mat. Mex. (3)}, 26(1):15--25, 2020.

\bibitem{Oh11}
Masahiro Ohtani.
\newblock Graphs and ideals generated by some 2-minors.
\newblock {\em Comm. Algebra}, 39(3):905--917, 2011.

\bibitem{R85}
Paul~C. Roberts.
\newblock A prime ideal in a polynomial ring whose symbolic blow-up is not
  {N}oetherian.
\newblock {\em Proc. Amer. Math. Soc.}, 94(4):589--592, 1985.

\bibitem{P88}
Peter Schenzel.
\newblock Filtrations and {N}oetherian symbolic blow-up rings.
\newblock {\em Proc. Amer. Math. Soc.}, 102(4):817--822, 1988.

\bibitem{AS99}
Anurag~K. Singh.
\newblock {$F$}-regularity does not deform.
\newblock {\em Amer. J. Math.}, 121(4):919--929, 1999.

\bibitem{VK21}
Vijaylaxmi Trivedi and Kei-Ichi Watanabe.
\newblock Hilbert-{K}unz density functions and {$F$}-thresholds.
\newblock {\em J. Algebra}, 567:533--563, 2021.

\end{thebibliography}

\end{document}